\documentclass[preprint,12pt, nopreprintline]{elsarticle}

\usepackage{amssymb,amsthm,amsmath,newtxtext,enumitem}
\usepackage[stix2]{newtxmath}
\usepackage{pgfplots}
\usepackage{float}
\pgfplotsset{compat=1.18}
\usepackage[colorlinks=true]{hyperref}

\newtheorem{theorem}{Theorem}[section]
\newtheorem{lemma}[theorem]{Lemma}

\newtheorem{proposition}[theorem]{Proposition}
\newtheorem{remark}[theorem]{Remark}

\newtheorem{problem}[theorem]{Problem}

\newcommand{\R}{\mathbb R}
\newcommand{\U}{\mathbf U}

\newcommand{\x}{\mathbf x}
\newcommand{\y}{\mathbf y}

\newcommand{\eps}{\varepsilon}

\newcommand{\dif}{\ensuremath{\,\mathrm{d}}}

\newcommand{\sgn}{\mathrm{sgn}\,}
\newcommand{\N}{\mathcal{N}}

\newcommand{\Lop}{\mathcal L}

\newcommand{\bxi}{\boldsymbol{\xi}}

\renewcommand{\div}{\mathop{\mathrm{div}}\nolimits}

\newcommand{\dist}{\mathop{\mathrm{dist}}\nolimits}

\numberwithin{equation}{section}
\allowdisplaybreaks[1]
\begin{document}
	
	\begin{frontmatter}
		
		
		
		\title{Steady Euler flows with contact discontinuities in infinitely long nozzles}
		
		\author[label1]{Jun Chen}
		\author[label1]{Xuemei Deng}
		\affiliation[label1]{organization={College of Science, China Three Gorges University},
			city={Yichang},
			state={Hubei},
			country={China}}
		
		\author[label2]{Xiaoguang You\corref{cor1}}
		\cortext[cor1]{Corresponding author. Email: xgyou@jxstnu.edu.cn}
		
		\affiliation[label2]{organization={School of Mathematical Sciences, Jiangxi Science \& Technology Normal University},
			city={Nanchang},
			state={Jiangxi},
			country={China}}
		
		\begin{abstract}
			We investigate two-dimensional steady subsonic Euler flows
			with a contact discontinuity in infinitely long nozzles.
			An Euler--Lagrange transformation reduces the problem to a scalar
			quasilinear elliptic equation in divergence form on a fixed strip,
			with discontinuous coefficients and the contact discontinuity
			conditions incorporated into the weak formulation.
			By solving this elliptic problem, we establish existence,
			uniqueness, and regularity of the flow on the maximal interval
			of wall amplitudes containing the straight-nozzle configuration.
			These results apply to general piecewise smooth subsonic upstream
			data and dispense entirely with the monotonicity, convexity,
			and small-perturbation assumptions imposed in earlier studies.
			In particular, the contact discontinuity curve is of
			class~$C^{2,\alpha}$.
			We further investigate the downstream asymptotic behavior
			and show, under the corresponding decay assumptions on the
			nozzle walls, that the convergence rate of the flow matches
			that of the walls.
		\end{abstract}

		
	\end{frontmatter}
	
	\section{Introduction}\label{Sec:introduction}
	
	We study the existence, uniqueness, and regularity of two-dimensional steady subsonic Euler flows with a contact discontinuity in infinitely long nozzles. In such a flow, the two subsonic streams are separated by a material curve across which the pressure is continuous, while the tangential velocity and entropy may jump. Such discontinuities include vortex sheets and entropy waves~\cite{Bae2013,Chen2019} and arise, for example, as slip lines in Mach reflection and as vortex sheets issuing from airfoil trailing edges~\cite{Chen2006,ChenXinZang2022}. For ducts whose length is large compared with their width, the infinite-nozzle formulation provides a model with prescribed upstream data and impermeable walls. In this setting, the flow and the contact curve must be determined together, along with their far-field behavior. These questions have motivated extensive work on steady nozzle problems.
	
	Early rigorous results on multidimensional transonic flows in infinite cylinders and nozzles were obtained by Chen and Feldman. They established the existence and stability of transonic shocks in infinite cylinders~\cite{Chen2004} and subsequently treated nozzles of arbitrary cross-sections~\cite{Chen2007}. In these problems, the shock separates a supersonic upstream flow from a subsonic downstream flow, and its location is determined together with the flow. These works, however, use the potential-flow model and do not address vorticity or contact discontinuities. To study nozzle flows with these features, we turn to the full compressible Euler equations. In this setting, Chen, Chen, and Feldman~\cite{Chen2007Transonic} established existence, stability, and asymptotic behavior of transonic flows in two-dimensional infinite nozzles of slowly varying cross-sections, for perturbations of a piecewise constant background state. For flows that remain subsonic throughout the nozzle, Xie and Xin~\cite{XieXin2010} proved global existence and uniqueness when the variation of the upstream Bernoulli function is sufficiently small and the mass flux lies in a suitable range. Chen, Deng, and Xiang~\cite{Chen2012Global} treated the full Euler system under smallness assumptions on the upstream entropy and Bernoulli oscillations in~$C^{1,1}$ and suitable bounds on the mass flux; they also established far-field asymptotics and studied the critical mass flux. Du, Xie, and Xin~\cite{DuXieXin2014} allowed large vorticity under a convexity condition on the upstream horizontal velocity. Related large-vorticity results for non-isentropic flows and axisymmetric nozzles were obtained by Chen~\cite{Chen2016} and Duan and Luo~\cite{DuanLuo2015}, respectively.
	
	Several existence and stability results for nozzle flows containing contact discontinuities have been obtained by perturbing prescribed background states. Bae~\cite{Bae2013} considered two uniform subsonic streams separated by a straight contact discontinuity and established existence, uniqueness, and stability under small perturbations of the duct walls. Bae and Park~\cite{BaePark2019SIAM,BaePark2019JDE} used Helmholtz decomposition to construct contact-discontinuity flows with nonzero vorticity in two-dimensional nozzles and three-dimensional axisymmetric cylinders, under small perturbations of the entrance data from piecewise constant background states. Weng and Zhang~\cite{WengZhang2025} studied subsonic contact-discontinuity flows in finitely long curved nozzles under small perturbations of background flows and the nozzle walls. Related results for subsonic flows past an airfoil with an attached contact discontinuity were obtained by Chen, Xin, and Zang~\cite{ChenXinZang2022}.
	
	Going beyond small perturbations of piecewise constant states, Chen, Huang, Wang, and Xiang~\cite{Chen2019} established smooth subsonic flows with large vorticity and weak flows containing contact discontinuities in arbitrary infinitely long nozzles. Their smooth-flow result removes the earlier convexity restrictions, while the result for discontinuous upstream data still requires monotonicity near the discontinuity and sign conditions at the inlet endpoints and across the jump. The contact curve is shown to be Lipschitz. They also obtained the subsonic--sonic and incompressible limits. Ma and Wang~\cite{MaWang2026} subsequently studied far-field convergence rates and structural stability for flows with large vorticity and contact discontinuities.
	
	In this paper, we remove the additional monotonicity, convexity, and sign conditions on the upstream data, and do not require these data to be close to piecewise constant states. For fixed~$\alpha\in(0,1)$ and piecewise~$C^{1,\alpha}$ upstream data satisfying the positivity and subsonicity conditions specified below, we establish existence, uniqueness, and regularity of subsonic flows throughout the nozzle for wall amplitudes in a maximal interval containing the straight-nozzle configuration. The flow is~$C^{1,\alpha}$ up to the contact curve from each side, and the curve is of class~$C^{2,\alpha}$. For any integer~$k\geqslant2$, if the upstream data are piecewise~$C^{k,\alpha}$ and the nozzle walls are~$C^{k+1,\alpha}$, then the flow is~$C^{k,\alpha}$ up to the contact curve from each side, and the curve is of class~$C^{k+1,\alpha}$. We also obtain a continuation criterion in terms of uniform subsonicity and a positive lower bound for the horizontal mass flux. Under the corresponding assumptions on the downstream wall profiles, we determine the limiting flow and its rate of convergence.
	
	We now formulate the nozzle problem and state the main results. The flow satisfies the two-dimensional steady compressible Euler system
	\begin{equation}\label{Equ:euler}
		\left\{
		\begin{aligned}
			&\div_{\mathbf x}(\rho \mathbf u)=0,\\
			&\div_{\mathbf x}(\rho \mathbf u\otimes \mathbf u)+\nabla_{\mathbf x}p=0,\\
			&\div_{\mathbf x}(\rho \mathbf u E+p\mathbf u)=0,
		\end{aligned}
		\right.
	\end{equation}
	where $\mathbf x=(x_1,x_2)\in\mathbb R^2$, $\rho$, $\mathbf u=(u_1,u_2)$, and $p$ denote the density, velocity, and pressure of the flow, respectively. For a polytropic gas,
	\[
	E=\frac{|\mathbf u|^2}{2}+\frac{p}{(\gamma-1)\rho},\qquad \gamma>1.
	\]
	For smooth solutions, the Bernoulli function and the entropy function are transported along streamlines:
	\begin{equation}\label{Equ:conservation}
		\div_{\mathbf x}(\rho\mathbf u B)=0,\qquad
		\div_{\mathbf x}(\rho\mathbf u S)=0,
	\end{equation}
	where
	\begin{equation}\label{Equ:Bernoulli}
		B=\frac12|\mathbf u|^2+\frac{\gamma p}{(\gamma-1)\rho},\qquad
		S=\frac{\gamma p}{(\gamma-1)\rho^\gamma}.
	\end{equation}
	The sound speed and the Mach number are given by
	\[
	c=\sqrt{\frac{\gamma p}{\rho}},\qquad M=\frac{|\mathbf u|}{c}.
	\]
	
	The superscripts $+$ and $-$ refer to the upper and lower regions, respectively. Let $W^+,W^-\in C^{2,\alpha}(\mathbb R)$ describe the nozzle walls. We assume that
	\begin{equation}\label{Cond:wall-regularity}
		\|W^\pm\|_{2,\alpha;\mathbb R}\leqslant C,
	\end{equation}
	and
	\begin{equation}\label{Def:omega}
		W^\pm(x_1)\to0\qquad\text{as }x_1\to-\infty.
	\end{equation}
	Let $\zeta\in\mathbb R$ be the wall amplitude, and define the infinitely long nozzle $\N_\zeta$ by
	\begin{equation}\label{Def:nozzle}
		\N_\zeta
		=
		\Big\{(x_1,x_2)\in\mathbb R^2\; \big| \;
		-1+\zeta W^-(x_1)<x_2<1+\zeta W^+(x_1)
		\Big\}.
	\end{equation}
	The upper and lower walls are
	\[
	\Gamma_\zeta^\pm
	=\big\{(x_1,x_2)\in\mathbb R^2:x_2=h_\zeta^\pm(x_1)\big\},
	\]
	where
	\[
	h_\zeta^+(x_1)=1+\zeta W^+(x_1),\qquad
	h_\zeta^-(x_1)=-1+\zeta W^-(x_1).
	\]
	On each wall we impose the slip boundary condition
	\begin{equation}\label{Cond:slip}
		\mathbf u\cdot\mathbf n^\pm=0\quad\text{on }\Gamma_\zeta^\pm,
	\end{equation}
	where $\mathbf n^+$ and $\mathbf n^-$ are the outward unit normals.
	Paired formulas containing $\pm$ are read with matching signs.
	We prescribe at the entrance $x_1=-\infty$ a piecewise smooth state
	\begin{equation}
		\mathbf U_\ell=(u_\ell,0,p_0,\rho_\ell),
	\end{equation}
	which has a jump discontinuity at some $x_2=x_2^*\in(-1,1)$.
	The transverse momentum equation requires the upstream pressure
	$p_0$ to be constant, and we assume $p_0>0$. We also assume that
	\begin{equation}\label{Cond:inlet-reg}
		u_\ell,\rho_\ell\in C^{1,\alpha}\big([-1,x_2^*]\big)\cap C^{1,\alpha}\big([x_2^*,1]\big),
	\end{equation}
	and
	\begin{equation}\label{Cond:inlet-pos}
		\inf_{[-1,x_2^*)\cup(x_2^*,1]} u_\ell>0,
		\quad \inf_{[-1,x_2^*)\cup(x_2^*,1]} \rho_\ell>0,
		\quad
		\sup_{[-1,x_2^*)\cup(x_2^*,1]}\frac{u_\ell}{c_\ell}<1,
	\end{equation}
	where
	\begin{equation}
		c_\ell = \sqrt{\frac{\gamma p_0}{\rho_\ell}}.
	\end{equation}
	The lower and upper mass fluxes are given by
	\begin{equation}\label{Def:m-pm}
		m^-=\int_{-1}^{x_2^*}\rho_\ell(s)u_\ell(s)\dif s,
		\qquad
		m^+=\int_{x_2^*}^{1}\rho_\ell(s)u_\ell(s)\dif s,
	\end{equation}
	and are positive by \eqref{Cond:inlet-pos}.
	
	We seek the flow together with an unknown contact discontinuity
	\[
	\Gamma_\zeta^*
	=
	\Big\{(x_1,x_2)\in \mathbb{R}^2 \ \big | \ x_2=\omega^*(x_1)\Big\},
	\qquad
	\omega^*(x_1)\to x_2^*
	\quad\text{as }x_1\to-\infty,
	\]
	which separates the nozzle into two subdomains
	\[
	\N_\zeta^+
	=
	\N_\zeta\cap\Big\{x_2>\omega^*(x_1)\Big\},
	\qquad
	\N_\zeta^-
	=
	\N_\zeta\cap\Big\{x_2<\omega^*(x_1)\Big\}.
	\]
	Let $\mathbf n$ be a unit normal to $\Gamma_\zeta^*$.
	The flow satisfies the contact discontinuity conditions
	\begin{equation}\label{Cond:RH-intro}
		\mathbf u \cdot \mathbf n = 0, \qquad [p] = 0,
	\end{equation}
	where, for a function $f$ with one-sided traces on $\Gamma_\zeta^*$, the jump $[f]$ is defined by
	\[
	[f](\mathbf x) := f|_{\N_\zeta^+}(\mathbf x) - f|_{\N_\zeta^-}(\mathbf x), \qquad \mathbf x \in \Gamma_\zeta^*.
	\]
	In this jump formula, $f|_{\N_\zeta^+}$ and $f|_{\N_\zeta^-}$
	denote the traces on $\Gamma_\zeta^*$ from the upper and lower
	regions, respectively.
	
	Since the flow is discontinuous across $\Gamma^*_\zeta$, we seek weak solutions of the Euler system \eqref{Equ:euler} in the following sense: for a locally integrable vector field $\mathbf F$ on a domain $\N$, the equation
	\[
	\operatorname{div} \mathbf F = 0
	\]
	is satisfied in the weak sense if
	\[
	\int_{\N} \mathbf F \cdot \nabla {v} \, \dif\mathbf x = 0 \qquad \text{for all } {v} \in C_c^\infty(\N).
	\]
	
	Throughout the paper, a superscript $\pm$ in a H\"older norm
	means the sum of the corresponding $+$ and $-$ norms:
	\[
	\begin{aligned}
		\|f\|_{k,\alpha;E^\pm}
		&:=\|f\|_{k,\alpha;E^+}+\|f\|_{k,\alpha;E^-},\\
		\|f^\pm\|_{k,\alpha;E}
		&:=\|f^+\|_{k,\alpha;E}+\|f^-\|_{k,\alpha;E}.
	\end{aligned}
	\]
	The first line applies to a function on two regions $E^+$ and
	$E^-$, and the second to two functions on the same set $E$.
	If several $\pm$ signs occur in a norm, they are matched before
	adding the two norms. Single-region norms are written with $+$
	or $-$ explicitly. Regularity up to the contact discontinuity
	is understood separately from each side, with the corresponding
	one-sided traces.
	
	We formulate the problem as follows.
	\begin{problem}\label{Problem:physical}
		Given $\zeta\in\mathbb R$ and upstream data $\mathbf U_\ell$
		satisfying \eqref{Cond:inlet-reg} and \eqref{Cond:inlet-pos}, find
		$\mathbf U=(\mathbf u,p,\rho)$ and $\omega^*$ with
		\[
		\begin{gathered}
			\mathbf U\in C^{1,\alpha}(\overline{\N_\zeta^+})\cap C^{1,\alpha}(\overline{\N_\zeta^-}),\quad  \omega^*\in C^{2,\alpha}(\R),
		\end{gathered}
		\]
		such that
		\begin{enumerate}[label=\textup{(\roman*)}, leftmargin=2em]
			\item $\mathbf U$ is a weak solution of \eqref{Equ:euler} subject to \eqref{Cond:slip} and \eqref{Cond:RH-intro};
			\item \label{Cond:p4} the flow is uniformly subsonic and has a positive lower bound for the horizontal mass flux:
			\[
			\sup_{\x \in\N^+_\zeta\cup\N^-_\zeta} M(\x)<1 \quad \text{ and } \quad \inf_{\x \in\N^+_\zeta\cup\N^-_\zeta} (\rho u_1)(\x)>0;
			\]
			\item for every $x_2\in(-1,x_2^*)\cup(x_2^*,1)$,
			\[
			\mathbf U(x_1,x_2)\to\mathbf U_\ell(x_2)\quad\text{as }x_1\to-\infty.
			\]
		\end{enumerate}
	\end{problem}
	
	The following theorem gives a unique solution on a maximal interval of wall amplitudes and describes the limiting behavior at finite endpoints.
	\begin{theorem}\label{Thm:main}
		There exist $\zeta^+\in\mathbb R_+\cup\{+\infty\}$ and
		$\zeta^-\in\mathbb R_-\cup\{-\infty\}$ such that, for each
		$\zeta\in(\zeta^-,\zeta^+)$, Problem~\ref{Problem:physical}
		admits a unique solution. Denote the corresponding flow by
		$\mathbf U^{(\zeta)}$ and its Mach number by $M^{(\zeta)}$.
		Set
		\begin{equation}
			\begin{aligned}
				d_{\rm s}(\zeta)
				&:=1-\sup_{\x\in\N_\zeta^+\cup\N_\zeta^-}
				M^{(\zeta)}(\x),\qquad
				d_{\rm m}(\zeta)
				:=\inf_{\x\in\N_\zeta^+\cup\N_\zeta^-}
				\bigl(\rho^{(\zeta)}u_1^{(\zeta)}\bigr)(\x).
			\end{aligned}
		\end{equation}
		If an endpoint $\zeta_*\in\{\zeta^-,\zeta^+\}$ is finite, then
		\begin{equation}\label{Est:blowup}
			\lim_{\substack{\zeta\to\zeta_*\\
					\zeta\in(\zeta^-,\zeta^+)}}
			\min\{d_{\rm s}(\zeta),d_{\rm m}(\zeta)\}=0.
		\end{equation}
	\end{theorem}
	
	For fixed $\eps>0$ and bounded wall amplitudes, the proof gives
	a uniform $C^{1,\alpha}$ bound for the flow whenever
	$\min\{d_{\rm s}(\zeta),d_{\rm m}(\zeta)\}\geqslant\eps$.
	The constant depends only on $\eps$, the bound on $|\zeta|$,
	and the prescribed data. Combined with local existence,
	this estimate gives a uniform continuation radius. In particular,
	if \eqref{Est:blowup} fails at a finite limiting parameter value
	$\zeta_*$, the solution branch extends uniquely to an open interval
	containing $\zeta_*$.
	
	If the upstream data and the nozzle walls possess higher regularity,
	then the solution has the corresponding higher regularity, as stated
	in the following theorem.
	
	\begin{theorem}\label{Thm:regularity}
		Let $k\geqslant2$ be an integer. Under the assumptions of
		Problem~\ref{Problem:physical}, assume further that
		\[
		\begin{gathered}
			\mathbf U_\ell\in C^{k,\alpha}([-1,x_2^*])\cap C^{k,\alpha}([x_2^*,1]),\quad W^\pm\in C^{k+1,\alpha}(\mathbb R).
		\end{gathered}
		\]
		Then, for every $\zeta\in(\zeta^-,\zeta^+)$, the solution
		$\mathbf U^{(\zeta)}$ in Theorem~\ref{Thm:main} belongs to
		$C^{k,\alpha}(\overline{\N_\zeta^+})\cap C^{k,\alpha}(\overline{\N_\zeta^-})$, and $\omega^*\in C^{k+1,\alpha}(\mathbb R)$.
	\end{theorem}
	
	\begin{remark}
		The existence theorem allows general downstream wall profiles satisfying \eqref{Cond:wall-regularity}. If both $W^+$ and $W^-$ have limits as $x_1\to+\infty$, the solution has a corresponding downstream asymptotic state, as proved in Section~\ref{Sec:outlet}.
	\end{remark}
	
	To construct the solution, we must determine the flow in the two subsonic regions together with the contact curve separating them. The normal mass flux vanishes on this curve, so it is a streamline. An Euler--Lagrange transformation therefore maps the nozzle to a fixed strip and the contact curve to a straight line; related uses of this transformation appear in~\cite{Bae2013,Chen2006,Chen2007Transonic}. The inverse stream function satisfies a scalar quasilinear elliptic equation in divergence form, with coefficients that may jump across the contact discontinuity. Continuity of the inverse stream function gives a common tangential trace, while the weak equation enforces pressure continuity. The contact conditions are thus incorporated into the elliptic problem without a separate iteration of the contact curve.
	
	The local construction and the continuation argument require different coefficient estimates. For local existence, we fix a solution and work in a small~$C^{2,\alpha}$ neighborhood of it. The coefficients are uniformly elliptic and have uniformly bounded piecewise~$C^{1,\alpha}$ norms throughout this neighborhood, although the bounds may depend on the chosen base solution. Boundary Schauder estimates near the two walls and the estimate across the contact discontinuity in~\cite{Chen2024Elliptic} give bounds independent of the truncation length. A fixed-point argument in the local H\"older topology then yields a solution for sufficiently small changes in the wall amplitude.
	
	For continuation, we instead derive the higher-order bound from the nondegeneracy conditions and the prescribed data, with the wall amplitude restricted to a bounded interval. Uniform subsonicity and a positive lower bound for the horizontal mass flux give uniform ellipticity and zero-order bounds for the nonlinear coefficient matrix. The tangential derivative of the inverse stream function and the pressure satisfy two scalar divergence-form equations weakly across the contact discontinuity. De Giorgi--Nash estimates, together with the algebraic relations, give a H\"older bound for the gradient near the contact curve. Near the nozzle walls, boundary H\"older estimates for the tangential derivative are combined with local energy estimates and the original equation to control the normal derivative. The resulting piecewise gradient bound controls the H\"older norms of the coefficients. Tangential Schauder estimates and recovery of the normal second derivative then yield a uniform~$C^{2,\alpha}$ bound and hence a uniform local continuation radius.
	
	The far-field analysis uses the unbounded strip geometry. Integral estimates establish the prescribed upstream limit. When the walls approach limiting downstream positions, translations and uniqueness identify the downstream profile. To obtain a convergence rate, we construct a barrier whose conormal derivative has the jump sign required by the weak maximum principle across the contact discontinuity. This gives the inverse stream function the same zero-order convergence rate as the wall profiles. Under the corresponding decay assumption on the wall~$C^{2,\alpha}$ norms, the Schauder estimates near the walls and the contact discontinuity yield the same rate for the derivatives. These downstream assumptions are needed only for the asymptotic results.
	
	The rest of the paper is organized as follows. Section~\ref{Sec:reformulation} introduces the Lagrangian coordinates and reduces the Euler problem to a scalar elliptic equation. Section~\ref{Sec:proof-main} establishes local existence and uniqueness, the upstream limit, and the conditional a priori estimates used for continuation. In Section~\ref{Sec:main}, we return to the physical coordinates and complete the proof of Theorem~\ref{Thm:main}. Section~\ref{Sec:regularity} proves the higher regularity of the flow and the contact curve. Section~\ref{Sec:outlet} establishes the downstream asymptotic state and the decay estimates.
	
	\section{Lagrangian formulation and elliptic reduction}\label{Sec:reformulation}
	
	Let $\mathbf U=(\mathbf u,p,\rho)$ be a solution of Problem~\ref{Problem:physical}. We use mass-flux coordinates to fix the nozzle walls and the contact curve, then derive a scalar elliptic equation for the inverse stream function. The return to physical coordinates is carried out in Section~\ref{Sec:main}.
	
	\subsection{The Euler--Lagrange transformation}
	
	Define
	\begin{equation}\label{Def:T}
		\y = (y_1,y_2)=\mathcal T(x_1,x_2):=\left(x_1,\int_{h_\zeta^-(x_1)}^{x_2}\rho u_1(x_1,s)\dif s-m^-\right).
	\end{equation}
	The mass equation and the zero normal mass flux on the walls and the contact curve preserve the fluxes $m^-$ and $m^+$ along the nozzle. Since $\rho u_1>0$, the map $\mathcal T$ sends $\N_\zeta$ onto the fixed strip
	\[
	\N := \mathbb R \times (-m^-, m^+),
	\]
	with upper and lower regions
	\[
	\N^+ := \mathbb R \times (0,m^+), \qquad
	\N^- := \mathbb R \times (-m^-,0).
	\]
	The contact discontinuity is mapped to the line
	\[
	\Gamma^* := \mathbb R \times \{0\},
	\]
	while the nozzle walls become
	\[
	\Gamma^+ := \mathbb R \times \{m^+\}, \qquad
	\Gamma^- := \mathbb R \times \{-m^-\}.
	\]
	The transformation preserves the upper and lower regions:
	$\mathcal T(\N_\zeta^+)=\N^+$ and $\mathcal T(\N_\zeta^-)=\N^-$.
	
	In Lagrangian coordinates, write
	$\widetilde{\mathbf U}(\y)=\mathbf U(\mathcal T^{-1}(\y))$.
	We keep the notation $\mathbf U$ for the transformed variables.
	As in \cite{Chen2007Transonic}, system~\eqref{Equ:euler} becomes
	\begin{equation}\label{Equ:lang}\left\{
		\begin{aligned}
			&\Big(\frac{1}{{\rho} u_1}\Big)_{y_1} - \Big(\frac{u_2}{u_1}\Big)_{y_2} = 0,\\
			&\Big(u_1 + \frac{ p}{{\rho} u_1}\Big)_{y_1} - \Big(\frac{ pu_2}{u_1}\Big)_{y_2} = 0,\\
			&(u_2)_{y_1} + {p}_{y_2} = 0,\\
			&\Big(\frac{1}{2}|\mathbf u|^2 + \frac{\gamma {p}}{(\gamma-1){\rho}}\Big)_{y_1} = 0,
		\end{aligned}\right.
	\end{equation}
	and the contact discontinuity condition \eqref{Cond:RH-intro} reads
	\begin{equation}\label{Cond:RH-y}
		\left[\frac{u_2}{u_1}\right] = [p] = 0 \quad \text{on } \Gamma^*,
	\end{equation}
	where $[f]$ denotes the upper trace minus the lower trace on $\Gamma^*$, as in \eqref{Cond:RH-intro}.
	
	At the upstream end, the change of transverse coordinate is
	\[
	y_2=\vartheta_\ell(x_2):=
	\int_{x_2^*}^{x_2}\rho_\ell(s)u_\ell(s)\dif s.
	\]
	The function $\vartheta_\ell$ is strictly increasing and maps $[-1,1]$ onto $[-m^-,m^+]$.
	In the Lagrangian variables we use the same symbols $u_\ell$ and
	$\rho_\ell$ for their compositions with $\vartheta_\ell^{-1}$.
	The transported quantities in \eqref{Equ:Bernoulli} then depend only
	on $y_2$; explicitly,
	\begin{equation}\label{Equ:SB}
		S(y_2) = \frac{\gamma p_0}{(\gamma-1)\rho_\ell(y_2)^\gamma},\qquad
		B(y_2) = \frac12 u_\ell(y_2)^2 + \frac{\gamma p_0}{(\gamma-1)\rho_\ell(y_2)}.
	\end{equation}
	Consequently, $S$ and $B$ are positive $C^{1,\alpha}$ functions on each of $[-m^-,0]$ and $[0,m^+]$, and may jump at $y_2=0$.
	
	\subsection{The inverse stream function}
	
	Write $\mathcal T^{-1}(y_1,y_2)=(y_1,\phi(y_1,y_2))$. Differentiating the transformation gives
	\begin{equation}\label{Def:velocity-Lag}
		\phi_{y_1}=\frac{u_2}{u_1},
		\qquad
		\phi_{y_2}=\frac1{\rho u_1}>0.
	\end{equation}
	The regularity of $\mathbf U$ in the two regions and
	condition~\ref{Cond:p4} in Problem~\ref{Problem:physical} give
	
	\begin{equation}\label{Est:varphi}
		\phi \in C^{2,\alpha}(\overline{\N^+})\cap C^{2,\alpha}(\overline{\N^-})\cap C(\overline{\N}).
	\end{equation}
	The slip boundary conditions on the nozzle walls become the Dirichlet conditions
	\begin{equation}\label{Cond:phi-wall}
		\phi=h_\zeta^\pm\quad\text{on }\Gamma^\pm.
	\end{equation}
	The upstream inverse stream function is
	\[
	\phi_\ell(y_2) = \int_{-m^-}^{y_2} \frac{1}{\rho_\ell(s)u_\ell(s)} \dif s - 1
	\quad \text{for } y_2 \in [-m^-, m^+].
	\]
	
	We next show that the fluid density can be expressed as a smooth
	function of $S$, $B$, and $\nabla\phi$. To this end, we first
	consider an auxiliary algebraic equation. For
	$\boldsymbol\eta\in(0,\infty)^2$ and
	$\boldsymbol\xi\in\mathbb R^2$ with $\xi_2>0$, define
	\begin{equation}\label{Def:chi}
		X(\boldsymbol{\xi}) = \frac{\xi_1^2+1}{2\xi_2^2} ,\qquad
		Q(\boldsymbol{\eta},\rho) = \eta_2\rho^2 - \eta_1\rho^{\gamma+1},
	\end{equation}
	where $\rho\geqslant0$ is an independent scalar variable.
	For fixed $\boldsymbol\eta$ and $\boldsymbol\xi$, consider the equation
	\begin{equation}\label{Equ:Q}
		X(\boldsymbol{\xi}) = Q(\boldsymbol{\eta},\rho).
	\end{equation}
	For fixed $\boldsymbol{\eta}$, the graph of $Q$ as a function of $\rho$ is shown in Figure~\ref{fig:Q_curve}.
	\begin{figure}[H]
		\centering
		\begin{tikzpicture}[
			declare function={
				gamma = 1.4;
				xi1 = 1;
				xi2 = 1;
				Q(\x) = xi2*\x^2 - xi1*\x^(gamma+1);
				rho_min_num = (2*xi2/(xi1*(gamma+1)))^(1/(gamma-1));
				Q_max_num = xi2*rho_min_num^2 - xi1*rho_min_num^(gamma+1);
				rho_max_num = (xi2/xi1)^(1/(gamma-1));
			},
			every node/.style={font=\tiny}
			]
			\begin{axis}[
				axis lines = middle,
				xlabel = $\rho$,
				ylabel = {$Q$},
				xmin = -0.1, xmax = rho_max_num + 0.1,
				ymin = -0.01, ymax = 0.08,
				domain = 0:rho_max_num,
				samples = 200,
				width = 12cm,
				height = 6cm,
				grid = none,
				xtick = \empty,
				ytick = \empty,
				tick style = {draw = none},
				axis line style = {thick},
				clip = false,
				xlabel style = {at={(ticklabel* cs:1)}, anchor=west, font=\normalsize},
				ylabel style = {at={(ticklabel* cs:1)}, anchor=south, font=\normalsize},
				]
				\addplot [thick, black] {Q(x)};
				
				\pgfmathsetmacro{\rmin}{rho_min_num}
				\pgfmathsetmacro{\rmax}{rho_max_num}
				
				
				\draw [black] (axis cs:\rmin,0) -- (axis cs:\rmin,-0.002);
				\node [below, align=center, font=\small] at (axis cs:\rmin,-0.002)
				{$\rho_{\rm cr}$};
				
				\draw [black] (axis cs:\rmax,0) -- (axis cs:\rmax,-0.002);
				\node [below, align=center, font=\small] at (axis cs:\rmax,-0.002)
				{$\rho^+$};
				
				\addplot [only marks, mark=*, black] coordinates {(\rmin, Q_max_num)};
				
				\draw [dashed, gray] (axis cs:\rmin, Q_max_num) -- (axis cs:0, Q_max_num);
				\node [left, align=right,font=\small] at (axis cs:-0.03, Q_max_num)
				{$Q_{\max}$
				};
				\draw [dashed, gray] (axis cs:\rmin, Q_max_num) -- (axis cs:\rmin, 0);
			\end{axis}
		\end{tikzpicture}
		\vspace{-0.5cm}
		\caption{Graph of $\rho\mapsto Q(\boldsymbol\eta,\rho)$ for $\eta_1=\eta_2=1$ and $\gamma=7/5$.}
		\label{fig:Q_curve}
	\end{figure}
	\noindent
	The maximum of $Q$ is attained at
	\begin{equation}\label{Def:rho-cr}
		\rho_{\rm cr}(\boldsymbol{\eta}) = \left(\frac{2 \eta_2}{(\gamma+1)\eta_1}\right)^{1/(\gamma-1)},
	\end{equation}
	with maximum value
	\begin{equation}\label{Def:Qmax}
		Q_{\max}(\boldsymbol{\eta})=\frac{\gamma-1}{\gamma+1}\eta_2\left(\frac{2\eta_2}{(\gamma+1)\eta_1}\right)^{2/(\gamma-1)}.
	\end{equation}
	Moreover, $Q$ vanishes at $\rho=0$ and at $\rho=\rho^+$, where
	\begin{equation}
		\rho^+(\boldsymbol{\eta}) = \left(\frac{\eta_2}{\eta_1}\right)^{1/(\gamma-1)}.
	\end{equation}
	Define the admissible domain
	\begin{equation}\label{Def:D}
		\mathcal D=\Big\{\big(\boldsymbol{\eta},\boldsymbol{\xi}\big)\in \mathbb R^4 \;\big|\; \eta_1>0,\ \eta_2>0,\ \xi_2>0,\ X(\boldsymbol{\xi})<Q_{\max}(\boldsymbol{\eta})\Big\}.
	\end{equation}
	For each $(\boldsymbol\eta,\boldsymbol\xi)\in\mathcal D$,
	equation~\eqref{Equ:Q} has a unique root in
	$(\rho_{\rm cr}(\boldsymbol\eta),\rho^+(\boldsymbol\eta))$.
	This choice of root defines the subsonic branch
	$\rho=\rho(\boldsymbol\eta,\boldsymbol\xi)$. On this branch,
	\begin{equation}
		\partial_\rho Q(\boldsymbol{\eta},\rho)=2\eta_2\rho-(\gamma+1)\eta_1\rho^\gamma<0.
	\end{equation}
	The implicit function theorem shows that $\rho(\boldsymbol\eta,\boldsymbol\xi)$ is smooth on $\mathcal D$.
	
	We now return to the flow. Bernoulli's law and
	\eqref{Def:velocity-Lag} show that the fluid density $\rho$
	under consideration satisfies \eqref{Equ:Q} with
	$(\boldsymbol\eta,\boldsymbol\xi)=(S,B,\nabla\phi)$, namely,
	\begin{equation}\label{Equ:X}
		X(\nabla \phi) = Q(S, B, \rho).
	\end{equation}
	Moreover, the subsonic condition $M<1$ yields
	\begin{equation}\label{Equ:subsonic-derivative}
		\partial_\rho Q(S,B,\rho)=2B\rho-(\gamma+1)S\rho^\gamma=\rho c^2(M^2-1)<0.
	\end{equation}
	Together with $Q(S,B,\rho)=X(\nabla\phi)>0$, this gives
	\begin{equation}
		\rho_{\rm cr}(S,B)<\rho<\rho^+(S,B)\quad\text{in }\N^\pm.
	\end{equation}
	Thus $(S,B,\nabla\phi)\in\mathcal D$, and the fluid density
	coincides with the root on the subsonic branch defined above.
	Consequently, all flow variables can be recovered from $\phi$ via
	\begin{equation}\label{Def:flow-recovery}
		\rho=\rho(S,B,\nabla\phi),\qquad
		u_1=\frac1{\rho\phi_{y_2}},\qquad
		u_2=\frac{\phi_{y_1}}{\rho\phi_{y_2}},\qquad
		p=\frac{\gamma-1}{\gamma}S\rho^\gamma.
	\end{equation}
	Moreover, there exists a constant $c_0>0$ such that
	\begin{equation}\label{Est:abc}
		\phi_{y_2}>\frac{1}{\sqrt{2Q_{\max}(S,B)}}\geqslant c_0.
	\end{equation}
	
	\subsection{Reformulation as an elliptic problem}
	
	For $(S,B,\boldsymbol\xi)\in\mathcal D$, let $\rho(S,B,\boldsymbol\xi)$ denote the root of \eqref{Equ:Q} on the subsonic branch. Define $\mathbf A=(A^1,A^2)$ by
	\begin{equation}\label{Def:A}
		\mathbf A(S,B,\boldsymbol{\xi})=\left(\frac{\xi_1}{\rho(S,B,\boldsymbol{\xi})\xi_2},\
		\frac{\gamma-1}{\gamma}S\rho(S,B,\boldsymbol{\xi})^\gamma\right).
	\end{equation}
	For any $\varphi\in C(\overline{\N})$ with $\varphi\in C^{2,\alpha}(\overline{\N^+})\cap C^{2,\alpha}(\overline{\N^-})$ and $(S,B,\nabla\varphi)\in\mathcal D$, we define the operator
	\begin{equation}
		\Lop(\varphi) = \div_{\y}\Big(\mathbf A\big(S(y_2),B(y_2),\nabla\varphi(\y)\big)\Big).
	\end{equation}
	Equivalently, in weak form,
	\begin{equation}\label{Def:Eop-weak}
		\langle\Lop(\varphi),v\rangle
		:=-\int_{\N}\mathbf A(S,B,\nabla\varphi)\cdot\nabla v\dif \y,
		\qquad v\in C_c^1(\N).
	\end{equation}
	
	The third equation in \eqref{Equ:lang}, the pressure continuity in \eqref{Cond:RH-y}, and the wall conditions \eqref{Cond:phi-wall} give the weak boundary value problem
	\begin{equation}\label{Equ:elliptic}
		\Lop(\phi)=0\quad\text{in }\N,\qquad
		\phi=h_\zeta^\pm\quad\text{on }\Gamma^\pm.
	\end{equation}
	To verify uniform ellipticity, we differentiate the flux in \eqref{Def:A}. The resulting coefficient matrix is symmetric, since
	\[
	A^1_{\xi_2}=A^2_{\xi_1}
	=-\frac{\rho c^2u_1u_2}{c^2-|\mathbf u|^2}.
	\]
	Moreover,
	\begin{equation}\label{Equ:A^1}
		A^1_{\xi_1}(S, B, \nabla \phi)
		= u_1 \frac{c^2 - u_1^2}{c^2 -|\mathbf u|^2}>0,
	\end{equation}
	and
	\begin{equation}\label{Equ:det-A}
		\det \mathrm D_{\bxi} \mathbf A(S, B,\nabla \phi)
		= \frac{c^2\rho^2u_1^4}{c^2-|\mathbf u|^2}>0.
	\end{equation}
	Condition~\ref{Cond:p4} gives uniform positive lower bounds in
	\eqref{Equ:A^1} and \eqref{Equ:det-A}. Together with the
	regularity in \eqref{Est:varphi}, it also bounds
	$\mathrm D_{\bxi}\mathbf A$ in both $\N^+$ and $\N^-$. Thus \eqref{Equ:elliptic} is uniformly elliptic.

	Condition~\ref{Cond:p4} also keeps $(S,B,\nabla\phi)$ uniformly
	away from $\partial\mathcal D$. Indeed, the fixed functions
	$S$ and $B$ have positive lower and finite upper bounds.
	The gradient is bounded, and \eqref{Est:abc} gives a positive
	lower bound for $\phi_{y_2}$. Uniform subsonicity keeps $\rho$
	away from $\rho_{\rm cr}$, so there is a constant $\kappa>0$ such that
	\[
	X(\nabla\phi)=Q(S,B,\rho)\leqslant Q_{\max}(S,B)-\kappa
	\quad\text{in }\N^\pm.
	\]
	These bounds place all values of $(S,B,\nabla\phi)$ in a
	compact subset of $\mathcal D$.
	
	For a function $\varphi$ that is continuously differentiable in
	both $\N^+$ and $\N^-$, define
	\begin{equation}\label{Def:K-zeta}
		\mathcal K[\varphi]
		:=\overline{\Big\{
			\big(S(y_2),B(y_2),\nabla\varphi(\y)\big)
			:\ \y\in\N^+\cup\N^-
			\Big\}}.
	\end{equation}
	For a fixed function $\phi$, we write $\mathcal K:=\mathcal K[\phi]$.
	For the family $\phi^{(\zeta)}$, we write
	$\mathcal K^{(\zeta)}:=\mathcal K[\phi^{(\zeta)}]$.
	The preceding argument shows that $\mathcal K\Subset\mathcal D$
	for the inverse stream function associated with a solution of
	Problem~\ref{Problem:physical}.
	
	The following lemma identifies $\mathcal K\Subset\mathcal D$ with the physical nondegeneracy conditions in~\ref{Cond:p4}.
	
	\begin{lemma}\label{Lm:K-equiv}
		Let $\phi\in C(\overline{\N})$ be a solution of \eqref{Equ:elliptic}, with $\phi\in C^{2,\alpha}(\overline{\N^+})\cap C^{2,\alpha}(\overline{\N^-})$, and let the corresponding flow variables $\mathbf U=(\mathbf u,p,\rho)$ be given by \eqref{Def:flow-recovery}. Then, in Lagrangian coordinates, condition~\ref{Cond:p4} is equivalent to $\mathcal K\Subset\mathcal D$.
	\end{lemma}
	
	\begin{proof}
		The forward implication has been established above. It remains to prove the converse. Suppose $\mathcal K\Subset\mathcal D$. By compactness, there exist positive constants $\kappa$ and $C$ such that
		\[
		X(\nabla\phi)\leqslant Q_{\max}(S,B)-\kappa,\qquad
		\phi_{y_2}\leqslant C\quad\text{in }\N^\pm.
		\]
		Consequently, $\rho$ is uniformly bounded away from $\rho_{\rm cr}$; in particular, $\sup_{\N^+\cup\N^-}M < 1$. Moreover, $\inf_{\N^+\cup\N^-}\rho u_1 \geqslant 1/C > 0$.
	\end{proof}

	Accordingly, instead of solving Problem~\ref{Problem:physical} directly, we consider the following fixed-boundary problem for $\phi$.
	\begin{problem}\label{Problem:elliptic}
		For given $\zeta \in \mathbb R$, find
		\[
		\phi\in C^{2,\alpha}(\overline{\N^+})\cap C^{2,\alpha}(\overline{\N^-})\cap C(\overline{\N}),
		\] such that:
		\begin{enumerate}[label=\textup{(\roman*)}, leftmargin=2em]
			\item \label{Cond:weak-solution} $\phi$ solves equation~\eqref{Equ:elliptic};
			\item \label{Cond:del-y-2} $\mathcal K[\phi]\Subset\mathcal D$;
			\item \label{Cond:convergence} the upstream asymptotic condition holds:
			\[
			\lim_{y_1\to -\infty}\|\nabla \phi(y_1,\cdot) - \nabla \phi_\ell\|_{L^\infty([-m^-,m^+])} =0.
			\]
		\end{enumerate}
	\end{problem}

	\section{Solvability of the elliptic problem}\label{Sec:proof-main}
	
	We solve the fixed-strip problem by local construction in the wall amplitude and a priori estimates for continuation. The resulting theorem yields Theorem~\ref{Thm:main} after the return to physical coordinates.
	\begin{theorem}\label{Thm:elliptic}
		There exist $\zeta^+\in\R_+\cup\{+\infty\}$ and
		$\zeta^-\in\R_-\cup\{-\infty\}$ such that
		Problem~\ref{Problem:elliptic} admits a unique solution
		$\phi^{(\zeta)}$ for each $\zeta\in(\zeta^-,\zeta^+)$.
		For $\zeta\in(\zeta^-,\zeta^+)$, set
		\begin{equation}\label{Est:stagnation}
			H(\zeta):=\|\partial_{y_2}\phi^{(\zeta)}\|_{0,0;\N^\pm}, \qquad
			d(\zeta):=\dist(\mathcal K^{(\zeta)},\partial\mathcal D).
		\end{equation}
		If an endpoint $\zeta_*\in\{\zeta^-,\zeta^+\}$ is finite, then
		\begin{equation}
			\lim_{\substack{\zeta\to\zeta_*\\\zeta\in(\zeta^-,\zeta^+)}}
			\min\left\{\frac{1}{1+H(\zeta)},\,d(\zeta)\right\}=0.
		\end{equation}
	\end{theorem}

	The proof uses two different sets of estimates. In the local
	construction near a fixed solution $\phi_0$, the coefficient
	bounds depend on $\|\phi_0\|_{2,\alpha;\N^\pm}$ and
	$\dist(\mathcal K[\phi_0],\partial\mathcal D)$. For continuation,
	Proposition~\ref{Pr:asymptotic-zeta} derives
	\eqref{Est:U-alpha} from \eqref{Equ:cond}, with
	$\overline C=\overline C(\eps,\zeta^+)$ depending only on
	$\eps$, the finite endpoint $\zeta^+$, and the fixed data.
	Thus the high-order norm $\|\phi^{(\zeta)}\|_{2,\alpha;\N^\pm}$ is controlled by the nondegeneracy bounds and the fixed data. This estimate makes the local continuation radius
	uniform as long as $H$ is bounded and $d$ has a fixed positive
	lower bound. The same argument applies at a finite left endpoint.
	
	We first establish local existence and uniqueness near any
	parameter value at which a solution exists. Starting from
	$\phi_\ell$ at $\zeta=0$, we obtain the maximal existence
	interval. We then prove the conditional a priori estimate
	and use the uniform continuation radius to establish the
	limit at each finite endpoint.
	
	Let
	\[
	\underline{m} := \min(m^-,m^+), \qquad \overline{m} := \max(m^-,m^+).
	\]
	For $R>0$, set
	\[
	\N_R := \N\cap\{|y_1|<R\}, \qquad
	\N_R^\pm := \N^\pm\cap\{|y_1|<R\},
	\]
	and
	\[
	\Gamma_R^\pm := \big\{(y_1,y_2)\in\Gamma^\pm : |y_1|<R\big\}, \qquad
	\Gamma_R^* := \big\{(y_1,y_2)\in\Gamma^* : |y_1|<R\big\}.
	\]

	For $q\in\mathbb R$, write
	\[
	\y_q^+=(q,m^+),\qquad \y_q^-=(q,-m^-),\qquad \y_q^*=(q,0).
	\]
	Define the wall half-balls
	\[
	D_r(\y_q^\pm):=B_r(\y_q^\pm)\cap\N^\pm,
	\qquad 0<r\leqslant m^\pm.
	\]
	For balls centered on the contact discontinuity, define
	\[
	B_r^\pm(\y_q^*):=B_r(\y_q^*)\cap\N^\pm,
	\qquad 0<r\leqslant\underline m.
	\]
	For other centers, the two regions are written explicitly as
	$B_r(\y)\cap\N^+$ and $B_r(\y)\cap\N^-$.
	Thus $D_r$ is used only near a wall, while $B_r^+$ and $B_r^-$
	are used only for the two parts of a ball centered on $\Gamma^*$.
	
	Fix $\zeta_0 \in \mathbb R$ such that Problem~\ref{Problem:elliptic} admits a solution $\phi_0$ at $\zeta = \zeta_0$. Let $\delta_0 > 0$ be a small constant to be determined later. For $-\delta_0 < \delta < \delta_0$, consider the problem \eqref{Equ:elliptic} with $\zeta = \zeta_0 + \delta$, which can be rewritten as
	\begin{equation}\label{Equ:diff-zeta-0}
		\begin{gathered}
			\partial_{y_i}\big(a_{ij}^{(\phi)}\partial_{y_j}\psi\big)
			=\Lop(\phi)-\Lop(\phi_0)=0\quad\text{in }\N,\qquad
			\psi=\delta g\quad\text{on }\Gamma^\pm,
		\end{gathered}
	\end{equation}
	where the increment $\psi$ and the averaged coefficients $a_{ij}^{(\phi)}$ are defined by
	\begin{equation}
		\psi = \phi - \phi_0, \qquad
		a_{ij}^{(\phi)} = \int_0^1 A^i_{\xi_{j}}\big(S, B, \nabla\phi_0 + s(\nabla \phi - \nabla \phi_0)\big) \,\dif s .
	\end{equation}
	The wall perturbation is extended to $\overline\N$ by
	\begin{equation}\label{Def:g}
		g(\mathbf y)=\chi(y_2)W^-(y_1)+\bigl(1-\chi(y_2)\bigr) W^+(y_1),
	\end{equation}
	where $\chi\in C^\infty(\mathbb R)$ satisfies $0\leqslant\chi\leqslant1$,
	$\chi=1$ near $y_2=-m^-$, and $\chi=0$ near $y_2=m^+$.
	Thus $g=W^\pm$ on $\Gamma^\pm$.

	\subsection{Coefficient bounds near a fixed solution}
	For $\eps>0$, define
	\begin{equation}
		\mathcal M_\eps := \Big\{
		\phi\in C(\overline{\N})\;\Big|\;
		\begin{array}{l}
			\phi\in C^{2,\alpha}(\overline{\N^+})\cap C^{2,\alpha}(\overline{\N^-}),\
			\|\phi-\phi_0\|_{2,\alpha;\N^\pm}\leqslant\eps
		\end{array}\Big\},
	\end{equation}
	and set
	\begin{equation}
		\mathcal K_\eps
		:=\overline{\bigcup_{\varphi\in\mathcal M_\eps}\mathcal K[\varphi]}.
	\end{equation}
	Since $\mathcal K^{(\zeta_0)}\Subset\mathcal D$, choosing $\eps>0$ sufficiently small gives $\mathcal K_\eps\Subset\mathcal D$. We fix $\phi_0$ and this neighborhood throughout the local construction. In particular,
	\[
	\|\phi\|_{2,\alpha;\N^\pm}
	\leqslant\|\phi_0\|_{2,\alpha;\N^\pm}+\eps,
	\qquad \phi\in\mathcal M_\eps.
	\]
	The coefficient estimates depend on the fixed base solution and the compact set $\mathcal K_\eps$, but are uniform for $\phi\in\mathcal M_\eps$. The following lemma is a consequence of the smoothness of
	$\mathbf A$ on $\mathcal D$.
	\begin{lemma}\label{Lm:N-elliptic}
		There exists $\eps>0$, depending on $\phi_0$, such that for all $\phi \in \mathcal M_\eps$,	the following inequalities
		\begin{equation}
			\partial_{\xi_1}A^1\big(S(y_2), B(y_2), \nabla \phi(\y)\big) > 0 , \quad \det \mathrm D_{\bxi}\mathbf A\big(S(y_2), B(y_2), \nabla \phi(\y)\big) > 0
		\end{equation}
		hold uniformly for $\y=(y_1,y_2)\in\N^\pm$. Moreover,
		\begin{equation}\label{Est:N}
			\|\mathrm D_{\bxi} \mathbf A\big(S, B, \nabla \phi\big) - \mathrm D_{\bxi} \mathbf A\big(S, B, \nabla \phi_0\big)\|_{1,\alpha;\N^\pm} \leqslant C\| \phi- \phi_0\|_{2,\alpha;\N^\pm},
		\end{equation}
		where $C$ depends on $S$, $B$, and $\nabla \phi_0$.
	\end{lemma}

	We next estimate the coefficients $a_{ij}^{(\phi)}$.
	
	\begin{proposition}\label{Pr:a-ij-uniform}
		Let $\eps$ be given by Lemma~\ref{Lm:N-elliptic}. There exist constants $0<\lambda\leqslant\Lambda<\infty$ such that for every $\phi\in \mathcal M_\eps$,
		\begin{equation}\label{Est:a-ij-uniform}
			\begin{gathered}
				a_{ij}^{(\phi)}e_i e_j\geqslant\lambda|\mathbf e|^2
				\quad\text{for all }\mathbf e\in\mathbb R^2,\qquad
				\max_{i,j\in\{1,2\}}
				\|a_{ij}^{(\phi)}\|_{1,\alpha;\N^\pm}\leqslant\Lambda.
			\end{gathered}
		\end{equation}
		Repeated indices $i,j\in\{1,2\}$ are summed in tensor
		contractions throughout the proof.
	\end{proposition}
	
	\begin{proof}
		Observe that $\phi_0+s(\phi-\phi_0)$ belongs to $\mathcal M_\eps$ for all $s\in[0,1]$. The compact inclusion $\mathcal K_\eps\Subset\mathcal D$ and the symmetry and positive definiteness of $\mathrm D_{\bxi}\mathbf A$ therefore give $\lambda>0$ such that
		\begin{equation}
			a_{ij}^{(\phi)}e_i e_j\geqslant\lambda|\mathbf e|^2.
		\end{equation}
		Moreover, from \eqref{Est:N} we deduce
		\begin{equation}
			\|a_{ij}^{(\phi)} - a_{ij}^{(\phi_0)}\|_{1,\alpha;\N^\pm}
			\leqslant C\| \phi -  \phi_0\|_{2,\alpha;\N^\pm}
			\leqslant C\eps.
		\end{equation}
		Choosing
		\[
		\Lambda=\max\left\{1,\lambda,\,
		\max_{i,j\in\{1,2\}}\|a_{ij}^{(\phi_0)}\|_{1,\alpha;\N^\pm}+C\eps\right\}
		\]
		gives \eqref{Est:a-ij-uniform}.
	\end{proof}
	
	\subsection{Uniform estimates on truncated domains}
	
	Fix $\phi\in\mathcal M_\eps$. We first solve the linear problem on $\N_R$ and estimate the solution away from its artificial ends, with constants independent of $R$:
	\begin{equation}\label{Equ:app}
		\partial_{y_i}\big(a_{ij}^{(\phi)} \partial_{y_j}\psi\big) = 0 \quad \text{in } \mathcal{N}_R,
		\qquad
		\psi = \delta g \quad \text{on } \partial \mathcal{N}_R.
	\end{equation}
	
	\begin{proposition}\label{Pr:app-estimate}
		Suppose $\phi\in\mathcal M_\varepsilon$ and
		$R>2\overline m$. Then \eqref{Equ:app} admits a unique weak
		solution $\psi\in H^1(\N_R)$ satisfying
		\begin{equation}\label{Est:psi-r}
			\|\psi\|_{2,\alpha;\N_{R-2\overline m}^\pm}
			\leqslant C|\delta|\,\|g\|_{2,\alpha;\N_R},
		\end{equation}
		where $C$ depends only on $m^+$, $m^-$, $\alpha$, and the
		coefficient bounds $\lambda,\Lambda$ in
		Proposition~\ref{Pr:a-ij-uniform}, and is independent of
		$R$ and of the choice of $\phi\in\mathcal M_\eps$.
	\end{proposition}
	
	\begin{proof}
		By Proposition~\ref{Pr:a-ij-uniform}, the coefficients
		$a_{ij}^{(\phi)}$ are symmetric and uniformly elliptic, and
		their $C^{1,\alpha}$ norms in $\N^+$ and $\N^-$ are bounded
		independently of $R$ and $\phi$. Standard variational theory
		gives a unique weak solution $\psi\in H^1(\N_R)$ with
		$\psi-\delta g\in H^1_0(\N_R)$; see Theorem~8.3 in
		\cite{Gilbarg1997Elliptic}. The weak maximum principle gives
		\begin{equation}\label{Est:app-Linfty}
			\|\psi\|_{0,0;\N_R}
			\leqslant|\delta|\,\|g\|_{0,0;\N_R}.
		\end{equation}
		
		For $0<s\leqslant+\infty$, with $\N_\infty:=\N$ and
		$\N_\infty^\pm:=\N^\pm$, define
		\[
		\begin{aligned}
			\mathcal U_s^+
			&:=\N_s\cap\{y_2>\underline m/4\},\\
			\mathcal U_s^-
			&:=\N_s\cap\{y_2<-\underline m/4\},\\
			\mathcal U_s^*
			&:=\N_s\cap\{|y_2|<\underline m/2\}.
		\end{aligned}
		\]
		These three strips cover $\N_s$, and the middle strip overlaps
		each wall strip in a strip of width $\underline m/4$.
		In the following two steps, we use this cover to combine the
		boundary Schauder estimates near the upper and lower nozzle
		walls with the Schauder estimate in \cite{Chen2024Elliptic}
		across the contact discontinuity.
		
		\medskip\noindent\textbf{Step 1. $C^{1,\alpha}$ estimate.} Let $s=R-\overline m$.
		We prove
		\begin{equation}\label{Est:app-C1}
			\|\psi\|_{1,\alpha;\N_{s}^\pm}
			\leqslant C|\delta|\,\|g\|_{1,\alpha;\N_R}.
		\end{equation}
		Fix $|q|<s$. The half-balls
		$D_{m^+}(\y_q^+)$ and $D_{m^-}(\y_q^-)$, as well as the ball
		$B_{\underline m}(\y_q^*)$, lie in $\N_R$.
		Applying the boundary Schauder estimate in the two half-balls
		and adding the resulting inequalities gives
		\begin{equation}\label{Est:app-wall-C1}
			\begin{aligned}
				&\|\psi\|_{1,\alpha;
					D_{m^\pm-\underline m/8}(\y_q^\pm)} \leqslant C\Big(
				\|\psi\|_{0,0;D_{m^\pm}(\y_q^\pm)}
				+|\delta|\,\|g\|_{1,\alpha;D_{m^\pm}(\y_q^\pm)}
				\Big).
			\end{aligned}
		\end{equation}
		Here we use Corollary~8.36 and the subsequent remark in
		\cite{Gilbarg1997Elliptic}.
		
		For the middle strip, the Schauder estimate in
		\cite{Chen2024Elliptic} gives
		\begin{equation}\label{Est:app-contact-C1}
			\|\psi\|_{1,\alpha;B_{3\underline m/4}^\pm(\y_q^*)}
			\leqslant C\|\psi\|_{0,0;B_{\underline m}(\y_q^*)}.
		\end{equation}
		The constants in \eqref{Est:app-wall-C1} and
		\eqref{Est:app-contact-C1} depend only on
		$m^+$, $m^-$, $\alpha$, $\lambda$, and $\Lambda$.
		
		The half-balls on the left-hand sides of
		\eqref{Est:app-wall-C1} and \eqref{Est:app-contact-C1}
		have radii exceeding the widths of the corresponding strips
		$\mathcal U_s^\pm$ and $\mathcal U_s^*\cap\N^\pm$
		by $\underline m/8$ and $\underline m/4$, respectively.
		As $q$ varies over $(-s,s)$, these half-balls cover the
		corresponding strips with a fixed overlap.
		The local estimates control the H\"older quotients for nearby
		points, while the uniform bound for $\nabla\psi$ controls the
		remaining pairs. Thus  \eqref{Est:app-Linfty}, \eqref{Est:app-wall-C1}, and
		\eqref{Est:app-contact-C1} yield
		\[
		\begin{aligned}
			&\|\psi\|_{1,\alpha;\mathcal U_s^\pm}
			+\|\psi\|_{1,\alpha;\mathcal U_s^*\cap\N^\pm}\leqslant
			C|\delta|\,\|g\|_{1,\alpha;\N_R},
		\end{aligned}
		\]
		with $C$ independent of $s$.
		Since the middle strip overlaps each wall strip by the fixed
		width $\underline m/4$, combining these strip estimates gives
		\eqref{Est:app-C1}.
		
		\medskip\noindent\textbf{Step 2. Tangential derivative estimate.}
		Let $s=R-2\overline m$ and set $w:=\partial_{y_1}\psi$. Differentiating \eqref{Equ:app} tangentially, justified by difference quotients, gives
		\begin{equation}\label{Equ:app-w}
			\begin{cases}
				\partial_{y_i}\bigl(a_{ij}^{(\phi)}\partial_{y_j}w\bigr)
				=\div\mathbf F
				&\text{in }\N_{R-\overline m},\\
				w=\delta g_{y_1}
				&\text{on }\Gamma_{R-\overline m}^\pm,
			\end{cases}
		\end{equation}
		where $\mathbf F=(F_1,F_2)$ and
		\[
		F_i:=-(\partial_{y_1}a_{ij}^{(\phi)})
		\partial_{y_j}\psi,
		\qquad i=1,2.
		\]
		The differentiated equation holds weakly across the contact discontinuity, with
		\[
		[w]=0,\qquad
		[a_{2j}^{(\phi)}\partial_{y_j}w-F_2]=0
		\quad\text{on }\Gamma^*_{R-\overline m},
		\]
		with the conormal equality interpreted weakly.
		Although $\partial_{y_1}a_{ij}^{(\phi)}$ contains second derivatives of $\phi$ and $\phi_0$, Proposition~\ref{Pr:a-ij-uniform} controls its $C^{0,\alpha}$ norm uniformly on $\mathcal M_\eps$. The product estimate therefore gives
		\[
		\begin{aligned}
			\|\mathbf F\|_{0,\alpha;\N_{R-\overline m}^\pm}
			&\leqslant C\max_{i,j\in\{1,2\}}
			\|\partial_{y_1}a_{ij}^{(\phi)}\|_{0,\alpha;\N^\pm}
			\|\psi\|_{1,\alpha;\N_{R-\overline m}^\pm}\\
			&\leqslant C\Lambda\,
			\|\psi\|_{1,\alpha;\N_{R-\overline m}^\pm}.
		\end{aligned}
		\]
		Combining this with Step~1 and $w=\psi_{y_1}$ yields
		\[
		\|w\|_{0,0;\N_{R-\overline m}^\pm}
		+\|\mathbf F\|_{0,\alpha;\N_{R-\overline m}^\pm}
		\leqslant C|\delta|\,\|g\|_{1,\alpha;\N_R}.
		\]
		
		For $|q|<s$, the half-balls $D_{m^\pm}(\y_q^\pm)$
		and the ball $B_{\underline m}(\y_q^*)$ lie in
		$\N_{R-\overline m}$, where the preceding bounds apply.
		The boundary Schauder estimates give
		\begin{equation}\label{Est:app-wall-w}
			\begin{aligned}
				&\|w\|_{1,\alpha;
					D_{m^\pm-\underline m/8}(\y_q^\pm)}\\
				&\qquad\qquad\leqslant C\Big(
				\|w\|_{0,0;D_{m^\pm}(\y_q^\pm)}
				+\|\mathbf F\|_{0,\alpha;D_{m^\pm}(\y_q^\pm)}
				+|\delta|\,\|g_{y_1}\|_{1,\alpha;
					D_{m^\pm}(\y_q^\pm)}
				\Big).
			\end{aligned}
		\end{equation}
		In the middle strip, the Schauder estimate in
		\cite{Chen2024Elliptic} gives
		\begin{equation}\label{Est:app-contact-w}
			\begin{aligned}
				&\|w\|_{1,\alpha;B_{3\underline m/4}^\pm(\y_q^*)}\leqslant C\Big(
				\|w\|_{0,0;B_{\underline m}(\y_q^*)}
				+\|\mathbf F\|_{0,\alpha;B_{\underline m}^\pm(\y_q^*)}
				\Big).
			\end{aligned}
		\end{equation}
		
		The half-balls on the left-hand sides of
		\eqref{Est:app-wall-w} and \eqref{Est:app-contact-w}
		have the same fixed excess over the corresponding strip
		widths as in Step~1. Their translates for $|q|<s$ therefore
		cover $\mathcal U_s^\pm$ and
		$\mathcal U_s^*\cap\N^\pm$.
		Using the preceding bounds for $w$ and $\mathbf F$, together
		with
		\[
		\|g_{y_1}\|_{1,\alpha;\N_R}
		\leqslant\|g\|_{2,\alpha;\N_R},
		\]
		we obtain from \eqref{Est:app-wall-w} and
		\eqref{Est:app-contact-w}
		\[
		\begin{aligned}
			&\|w\|_{1,\alpha;\mathcal U_s^\pm}
			+\|w\|_{1,\alpha;\mathcal U_s^*\cap\N^\pm}\leqslant
			C|\delta|\,\|g\|_{2,\alpha;\N_R}.
		\end{aligned}
		\]
		The overlap of width $\underline m/4$ between the middle
		strip and each wall strip then gives
		\begin{equation}\label{Est:w-C1}
			\|w\|_{1,\alpha;\N_{R-2\overline m}^\pm}
			\leqslant C|\delta|\,\|g\|_{2,\alpha;\N_R}.
		\end{equation}
		Since $w=\psi_{y_1}$, it follows that
		\begin{equation}\label{Est:psi-11-12}
			\|\partial_{y_i}\partial_{y_j}\psi\|_{0,\alpha;
				\N_{R-2\overline m}^\pm}
			\leqslant C|\delta|\,\|g\|_{2,\alpha;\N_R},
			\qquad (i,j)\ne(2,2).
		\end{equation}
		
		\medskip\noindent\textbf{Step 3. Normal second derivative estimate.}
		In each of $\N_{R-2\overline m}^+$ and $\N_{R-2\overline m}^-$, expanding \eqref{Equ:app} and using
		the symmetry of the coefficients gives
		\begin{equation}\label{Equ:psi-22}
			a_{22}^{(\phi)}\psi_{y_2y_2}
			=-a_{11}^{(\phi)}\psi_{y_1y_1}
			-2a_{12}^{(\phi)}\psi_{y_1y_2}
			-(\partial_{y_i}a_{ij}^{(\phi)})\partial_{y_j}\psi.
		\end{equation}
		By Steps~1 and~2, the right-hand side is bounded in
		$C^{0,\alpha}$ on each of these two regions by
		$C|\delta|\,\|g\|_{2,\alpha;\N_R}$.
		Since $a_{22}^{(\phi)}\geqslant\lambda$ and its
		$C^{0,\alpha}$ norm is uniformly bounded, division by
		$a_{22}^{(\phi)}$ yields
		\[
		\|\psi_{y_2y_2}\|_{0,\alpha;\N_{R-2\overline m}^\pm}
		\leqslant C|\delta|\,\|g\|_{2,\alpha;\N_R}.
		\]
		Combining this with \eqref{Est:app-C1} and
		\eqref{Est:psi-11-12} proves \eqref{Est:psi-r}.
	\end{proof}
	
	\subsection{Compactness and fixed-point argument}
	For fixed $\phi\in\mathcal M_\eps$, we pass to the infinite strip using the estimates of Proposition~\ref{Pr:app-estimate}. The limiting linear problem is
	\begin{equation}\label{Equ:app-phi}
		\partial_{y_i}\big(a_{ij}^{(\phi)}\partial_{y_j}\psi\big)=0
		\quad\text{in }\N,\qquad
		\psi=\delta g\quad\text{on }\Gamma^\pm.
	\end{equation}
	
	\begin{proposition}\label{Pr:linear}
		Suppose $\phi \in \mathcal M_\eps$. Then \eqref{Equ:app-phi} admits a unique solution
		\[
		\psi^{(\phi)}\in C^{2,\alpha}(\overline{\N^+})\cap C^{2,\alpha}(\overline{\N^-})\cap C(\overline{\N}).
		\] Moreover,
		\begin{equation}\label{Est:psi-2alpha}
			\|\psi^{(\phi)}\|_{2,\alpha;\N^\pm} \leqslant C |\delta|\|g\|_{2,\alpha;\N}.
		\end{equation}
	\end{proposition}
	
	\begin{proof}
		\medskip\noindent\textbf{Step 1. Existence.}
		For $R>2\overline{m}$, let $\psi^{R}\in H^1(\N_{R})$ be the unique solution of \eqref{Equ:app} satisfying
		\begin{equation}\label{Est:psi-m}
			\|\psi^{R}\|_{2,\alpha;\N_{R-2\overline{m}}^\pm}\leqslant C|\delta|\|g\|_{2,\alpha;\N_{R}} \leqslant C|\delta|\|g\|_{2,\alpha;\N}.
		\end{equation}
		By the Arzel\`a--Ascoli theorem and a diagonal extraction,
		there is a sequence $R\to\infty$ along which $\psi^R$ converges to a function
		$\psi^{(\phi)}\in H^1_{\mathrm{loc}}(\N)$ such that
		\[
		\begin{gathered}
			\psi^R\to\psi^{(\phi)}\quad\text{in }C_{\mathrm{loc}}(\overline\N) \cap 
			C^{2}_{\mathrm{loc}}(\overline{\N^+})\cap C^{2}_{\mathrm{loc}}(\overline{\N^-}).
		\end{gathered}
		\]
		Passing to the weak equation and the wall traces shows that
		$\psi^{(\phi)}$ solves \eqref{Equ:app-phi}. For every fixed
		$L>0$, the estimate \eqref{Est:psi-m} passes to the limit and gives
		\begin{equation}\label{Est:psi-tmp}
			\|\psi^{(\phi)}\|_{2,\alpha;\N_L^\pm}
			\leqslant C|\delta|\,\|g\|_{2,\alpha;\N}.
		\end{equation}
		Letting $L\to\infty$ proves \eqref{Est:psi-2alpha}.
		
		\medskip\noindent\textbf{Step 2. Uniqueness.}
		Let $\psi^1$ and $\psi^2$ be bounded weak solutions with the same
		wall data, and put $U:=\psi^1-\psi^2$.
		Choose $\eta_L\in C_c^\infty(\mathbb R)$ with
		$\eta_L=1$ on $[-L,L]$, $\eta_L=0$ outside $[-2L,2L]$, and
		$|\eta_L'|\leqslant C/L$. Testing the difference equation with
		$\eta_L^2U$ and applying Young's inequality gives
		\begin{equation}\label{Est:L}
			\int_{\N_L}|\nabla U|^2
			\leqslant C\int_{\N}|\eta_L'|^2|U|^2
			\leqslant \frac{C}{L}\|U\|_{L^\infty(\N)}^2.
		\end{equation}
		Letting $L\to\infty$ gives $\nabla U=0$, and the zero wall trace yields $U=0$.
	\end{proof}

	We next show that equation~\eqref{Equ:diff-zeta-0} is uniquely solvable.
	\begin{proposition}\label{Pr:existence-varphi}
		There exists $\delta_0>0$ such that for any $\zeta \in (\zeta_0-\delta_0,\zeta_0+\delta_0)$, equation~\eqref{Equ:diff-zeta-0} admits a unique solution
		\[
		\phi=\phi^{(\zeta)}\in C^{2,\alpha}(\overline{\N^+})\cap C^{2,\alpha}(\overline{\N^-})\cap C(\overline{\N}),
		\]
		with $\mathcal K[\phi]\Subset\mathcal D$. Moreover,
		\begin{equation}\label{Est:varphi-2alpha}
			\| \phi^{(\zeta)}- \phi_0\|_{2,\alpha;\N^\pm} \leqslant C |\zeta - \zeta_0|\|g\|_{2,\alpha;\N}.
		\end{equation}
	\end{proposition}
	\begin{proof}
		By \eqref{Est:psi-2alpha}, we may choose $\delta_0>0$ so that
		\[
		\|\psi^{(\phi)}\|_{2,\alpha;\N^\pm}\leqslant\eps
		\qquad(\phi\in\mathcal M_\eps,\ |\delta|<\delta_0).
		\]
		Hence
		\[
		\mathcal F(\phi):=\phi_0+\psi^{(\phi)}
		\]
		maps $\mathcal M_\eps$ into itself.
		
		We equip this set with the local H\"older topology. Let
		\[
		\mathscr X:=\left\{v\in C(\overline\N):
		\begin{array}{l}
			v|_{\overline{\N^+}}\in C^{2,\alpha/2}_{\mathrm{loc}}(\overline{\N^+}),\quad
			v|_{\overline{\N^-}}\in C^{2,\alpha/2}_{\mathrm{loc}}(\overline{\N^-})
		\end{array}\right\},
		\]
		with seminorms
		\[
		p_j(v):=\|v\|_{2,\alpha/2;\N_j^\pm},\qquad j\in\mathbb N.
		\]
		This is a Hausdorff locally convex metrizable space. The
		continuity requirement identifies the two traces on $\Gamma^*$.
		
		The set $\mathcal M_\eps$ is nonempty and convex. Let
		$\{v_k\}\subset\mathcal M_\eps$. The global $C^{2,\alpha}$ bound and the
		compact embedding on each bounded substrip
		$\N_j=\N\cap\{|y_1|<j\}$ give a subsequence converging in
		$C^{2,\alpha/2}(\overline{\N_j^+})\cap C^{2,\alpha/2}(\overline{\N_j^-})$.
		A diagonal extraction gives convergence to a limit $v$ in $\mathscr X$.
		The two one-sided traces coincide in the limit. Moreover, the global H\"older
		norm is lower semicontinuous under this convergence: each
		derivative value and each two-point H\"older quotient can be
		passed to the limit in a sufficiently large bounded substrip.
		Thus
		\[
		\|v-\phi_0\|_{2,\alpha;\N^\pm}
		\leqslant\liminf_{k\to\infty}
		\|v_k-\phi_0\|_{2,\alpha;\N^\pm}\leqslant\eps.
		\]
		Consequently, $\mathcal M_\eps$ is compact in $\mathscr X$.
		
		We next prove continuity of $\mathcal F$ in this topology.
		Let $\{\phi_k\}\subset\mathcal M_\eps$ converge to
		$\phi\in\mathcal M_\eps$ in $\mathscr X$, and set
		$\psi_k:=\mathcal F(\phi_k)-\phi_0$,
		$\psi:=\mathcal F(\phi)-\phi_0$.
		By \eqref{Est:psi-2alpha}, every subsequence of $\{\psi_k\}$ has
		a further subsequence converging in $\mathscr X$ to a function
		$\psi^*$ with $\psi^*+\phi_0\in\mathcal M_\eps$.
		All arguments of $\mathrm D_{\bxi}\mathbf A$ in the coefficients
		$a_{ij}^{(\phi_k)}$ and $a_{ij}^{(\phi)}$ belong to the fixed
		compact set $\mathcal K_\eps\Subset\mathcal D$.
		The smoothness of $\mathbf A$ on $\mathcal D$ therefore gives
		\[
		a_{ij}^{(\phi_k)}\to a_{ij}^{(\phi)}
		\quad\text{in }
		C^{1,\alpha/2}_{\mathrm{loc}}(\overline{\N^+})
		\cap C^{1,\alpha/2}_{\mathrm{loc}}(\overline{\N^-}).
		\]
		Passing to the limit in the weak equation against a compactly
		supported test function and in the wall traces shows that
		\[
		\partial_{y_i}(a_{ij}^{(\phi)}\partial_{y_j}\psi^*)=0
		\quad\text{in }\N,\qquad
		\psi^*=\delta g\quad\text{on }\Gamma^\pm.
		\]
		The same problem is satisfied by $\psi$. Uniqueness from
		Proposition~\ref{Pr:linear} implies $\psi^*=\psi$.
		Relative compactness and uniqueness of the subsequential limit imply that $\mathcal F(\phi_k)\to\mathcal F(\phi)$ in $\mathscr X$. The Schauder--Tychonoff fixed-point theorem
		\cite{Tychonoff1935} now gives a fixed point of $\mathcal F$.
		It solves \eqref{Equ:diff-zeta-0} and satisfies
		\eqref{Est:varphi-2alpha}.
		
		To prove uniqueness, let $\phi^1$ and $\phi^2$ be two solutions
		with $\mathcal K[\phi^i]\Subset\mathcal D$ for $i=1,2$, and put
		$U:=\phi^2-\phi^1$.
		Then $U=0$ on $\Gamma^\pm$, and
		\[
		\partial_{y_i}(a_{ij}\partial_{y_j}U)=0,\qquad
		a_{ij}:=\int_0^1 A^i_{\xi_j}
		(S,B,\nabla\phi^1+s(\nabla\phi^2-\nabla\phi^1))\dif s.
		\]
		For fixed $S,B>0$, the set
		\[
		\{\boldsymbol\xi\in\mathbb R^2:(S,B,\boldsymbol\xi)\in\mathcal D\}
		=\left\{(\xi_1,\xi_2):
		\xi_2>\frac{\sqrt{1+\xi_1^2}}{\sqrt{2Q_{\max}(S,B)}}\right\}
		\]
		is convex. Since $\mathcal K[\phi^i]\Subset\mathcal D$ for $i=1,2$,
		the closure of the set of points
		\[
		\big(S(y_2),B(y_2),
		(1-s)\nabla\phi^1(\y)+s\nabla\phi^2(\y)\big),
		\quad s\in[0,1],\quad \y\in\N^+\cup\N^-,
		\]
		is a compact subset of $\mathcal D$. Indeed, each limit point
		is obtained by interpolating two points of
		$\mathcal K[\phi^1]$ and $\mathcal K[\phi^2]$ with the same
		first two coordinates. The convexity just proved places it
		in $\mathcal D$. The averaged matrix $(a_{ij})$ is therefore
		bounded and uniformly elliptic. Applying \eqref{Est:L} to $U$
		gives $U=0$.
		
		Finally, the choices of $\eps$ and $\delta_0$ can be made
		uniformly when the base solutions $\phi_0$ have a common
		$C^{2,\alpha}$ bound and the sets $\mathcal K[\phi_0]$ are
		contained in one compact subset of $\mathcal D$.
		Their distances from $\partial\mathcal D$ then have a common
		positive lower bound. This gives a uniform choice of $\eps$,
		uniform coefficient bounds in Proposition~\ref{Pr:a-ij-uniform},
		and a uniform constant in \eqref{Est:psi-2alpha}, hence a
		uniform choice of $\delta_0$.
		Proposition~\ref{Pr:asymptotic-zeta} establishes the required
		common $C^{2,\alpha}$ bound and compact containment from
		\eqref{Equ:cond}. We then apply the present local construction
		with a uniform $\delta_0$ near a finite endpoint.
	\end{proof}
	
	The solution in Proposition~\ref{Pr:existence-varphi} satisfies
	conditions~\ref{Cond:weak-solution} and~\ref{Cond:del-y-2} in
	Problem~\ref{Problem:elliptic}.
	
	\subsection{Upstream asymptotic behavior}
	
	We verify the upstream condition by showing that the local solution has the same limiting gradient as the fixed base solution. The wall perturbation tends to zero upstream, and an energy estimate on long substrips transfers this decay to the gradient.

	\begin{proposition}\label{Pr:asymptotic}
		Let $\phi \in \mathcal M_\eps$ be the solution constructed in Proposition~\ref{Pr:existence-varphi}. Then
		\begin{equation}
			\lim_{y_1\to -\infty}\|\nabla \phi(y_1,\cdot)-\nabla \phi_\ell\|_{L^\infty([-m^-,m^+])}=0.
		\end{equation}
	\end{proposition}
	
	\begin{proof}
		
		Recall that $\psi=\phi-\phi_0$ solves
		\eqref{Equ:diff-zeta-0}. For $L>0$, define
		\begin{equation}
			\ell_L =
			\Big(
			\max\left\{
			\|W^+\|_{L^\infty(-\infty,-L/2)}^{1/2},\,
			\|W^-\|_{L^\infty(-\infty,-L/2)}^{1/2}\right\}
			+ (L/2)^{-1} \Big)^{-1}.
		\end{equation}
		By \eqref{Def:omega}, $\ell_L\to\infty$ as $L\to\infty$.
		Also, $\ell_L\leqslant L/2$, so the cutoff chosen below is
		supported in $\{y_1\leqslant-L/2\}$.
		
		Let $\eta=\eta(y_1)$ be a smooth cutoff with $0\leqslant\eta\leqslant1$ and
		\begin{equation}
			\begin{cases}
				\eta = 1 & \text{for } y_1 \in [-L-1,-L],\\[2pt]
				\eta = 0 & \text{for } y_1 \leqslant -L-1-\ell_L \ \text{or } y_1 \geqslant -L+\ell_L,\\[2pt]
				|\nabla \eta| \leqslant 2\ell_L^{-1} & \text{in } \N.
			\end{cases}
		\end{equation}
		Integrating by parts in the two regions, with cancellation of
		the conormal terms on the contact discontinuity, gives
		\begin{equation}
			\begin{aligned}
				\int_{\N} a_{ij}^{(\phi)} \,\partial_{y_j}\psi \,\partial_{y_i}\psi \,\eta^2
				&=
				\Big(\int_{\Gamma^+}-\int_{\Gamma^-}\Big)
				a_{2j}^{(\phi)} \,\partial_{y_j}\psi \,\psi\,\eta^2 \\
				&\quad
				-2 \int_{\N}
				a_{ij}^{(\phi)} \,\partial_{y_j}\psi \, \psi\, \partial_{y_i}\eta \,\eta.
			\end{aligned}
		\end{equation}
		
		Uniform ellipticity yields
		\[
		a_{ij}^{(\phi)} \,\partial_{y_i}\psi \,\partial_{y_j}\psi
		\geqslant \lambda |\nabla \psi|^2
		\quad \text{in } \N^+\cup\N^-.
		\]
		Combining this with the bounds
		\[
		\big|a_{2j}^{(\phi)} \,\partial_{y_j}\psi \,\psi\,\eta^2 \big| \leqslant C|\delta||g|
		\leqslant C \ell_L^{-2}
		\quad \text{on } \Gamma^\pm \cap \{y_1 < -L/2\},
		\]
		and
		\[
		2\big|a_{ij}^{(\phi)} \,\partial_{y_j}\psi \, \psi\,\partial_{y_i}\eta \,\eta\big|
		\leqslant \frac{\lambda}{2} \eta^2|\nabla \psi|^2 + C \ell_L^{-2},
		\]
		and using that $\eta$ is supported in an interval of length $1+2\ell_L$, we obtain, for all sufficiently large $L$,
		\begin{equation}\label{Est:decay-l}
			\int_{[-L-1,-L]\times[-m^-,m^+]} |\nabla \psi|^2
			\leqslant C \ell_L^{-1},
		\end{equation}
		where $C$ is independent of $L$.
		
		The estimate \eqref{Est:decay-l} implies
		\begin{equation}
			\lim_{y_1\to -\infty}
			\|\nabla \psi(y_1,\cdot)\|_{L^\infty([-m^-,m^+])}
			= 0.
		\end{equation}
		Indeed, if the asserted convergence fails, it fails in at least
		one of the two regions. We treat $\N^+$; the argument in $\N^-$
		is identical. There exist $\tau>0$ and a sequence
		$\y^n=(y_1^n,y_2^n)\in\N^+$ with $y_1^n\to-\infty$
		such that
		\[
		|\nabla\psi(\y^n)|>\tau.
		\]
		The uniform $C^{2,\alpha}$ bound on $\overline{\N^+}$
		gives $0<r<\min\{1,\underline m\}$, independent of $n$, such that
		\[
		|\nabla\psi(\y)|>\tau/2
		\quad\text{for }\y\in B_r(\y^n)\cap\N^+.
		\]
		Hence
		\begin{equation}
			\int_{([y_1^n-1,y_1^n]\times\mathbb R)\cap\N^+}
			|\nabla\psi|^2\geqslant c\tau^2r^2,
		\end{equation}
		with $c>0$ independent of $n$. Taking $L=-y_1^n$ in
		\eqref{Est:decay-l} gives a contradiction.
		Since $\phi=\phi_0+\psi$ and $\phi_0$ satisfies the upstream
		condition, the proposition follows.
		
	\end{proof}

	\subsection{Continuation to finite endpoints}
	
	Propositions~\ref{Pr:existence-varphi} and~\ref{Pr:asymptotic} give local existence and uniqueness in the wall amplitude, including the upstream condition. Let $I=(\zeta^-,\zeta^+)$ be the maximal open interval containing $0$ on which a solution exists. It is nonempty because $\phi_\ell$ solves the problem at $\zeta=0$.
	
	We prove the limit at a finite right endpoint $\zeta^+$.
	The same argument applies at a finite left endpoint.
	
	\begin{proposition}\label{Pr:asymptotic-zeta}
		If $\zeta^+<\infty$, then, with $H$ and $d$ defined in
		\eqref{Est:stagnation},
		\begin{equation}\label{Est:del-2-phi}
			\lim_{\substack{\zeta\to\zeta^+\\\zeta\in(\zeta^-,\zeta^+)}}
			\min\left\{\frac{1}{1+H(\zeta)},\,d(\zeta)\right\}=0.
		\end{equation}
	\end{proposition}
	
	To prove Proposition~\ref{Pr:asymptotic-zeta}, it suffices to
	show that for each $\eps>0$ there exists
	$\overline C=\overline C(\eps,\zeta^+)$, independent of
	$\zeta\in[0,\zeta^+)$, such that whenever
	\begin{equation}\label{Equ:cond}
		H(\zeta)<\frac1\eps
		\quad\text{and}\quad d(\zeta)>\eps,
	\end{equation}
	one has
	\begin{equation}\label{Est:U-alpha}
		\|\phi^{(\zeta)}\|_{2,\alpha;\N^\pm}\leqslant\overline C.
	\end{equation}
	Here $\overline C$ depends only on $\eps$, $\zeta^+$, and the
	fixed data. The existing $C^{2,\alpha}$ regularity justifies the computations
	for each solution. The argument below derives a common bound
	from \eqref{Equ:cond}.
	To see why \eqref{Est:U-alpha} suffices, suppose that \eqref{Est:del-2-phi} fails. Then there are $\eps_0>0$ and a sequence $\zeta_n\in(\zeta^-,\zeta^+)$ with
	$\zeta_n\to\zeta^+$ such that
	\begin{equation}\label{Est:K}
		\min\left\{\frac{1}{1+H(\zeta_n)},\,d(\zeta_n)\right\}
		\geqslant2\eps_0\qquad\text{for all }n.
	\end{equation}
	After discarding finitely many terms, we have $0\leqslant\zeta_n<\zeta^+$.
	Both inequalities in \eqref{Equ:cond} hold at every
	$\zeta_n$ with $\eps=\eps_0$.
	The estimate \eqref{Est:U-alpha}, together with the containment
	of all sets $\mathcal K^{(\zeta_n)}$ in a common compact subset
	of $\mathcal D$, gives a uniform local continuation radius
	$\delta_0>0$ by the last paragraph of
	Proposition~\ref{Pr:existence-varphi}.
	The upstream argument in Proposition~\ref{Pr:asymptotic} applies
	to each of these extensions. For $n$ sufficiently large,
	$\zeta_n+\delta_0>\zeta^+$. Uniqueness identifies the extension
	with the existing branch on the overlap and continues it beyond
	$\zeta^+$, contradicting maximality.
	
	We prove \eqref{Est:U-alpha} for a fixed $\eps>0$.
	In the rest of this subsection, assume $0\leqslant\zeta<\zeta^+<\infty$
	and write $\phi=\phi^{(\zeta)}$, with \eqref{Equ:cond} in force.
	Bernoulli's law \eqref{Equ:X} and the bound on $H(\zeta)$ give
	a uniform bound for $\phi_{y_1}$. Hence there is a compact set
	$\mathcal K(\eps)\Subset\mathcal D$, independent of $\zeta$, such that
	\begin{equation}
		\mathcal K^{(\zeta)}\subset\mathcal K(\eps).
	\end{equation}
	Also, since $\phi$ is increasing in $y_2$ and has the prescribed
	wall traces, $\|\phi\|_{0,0;\N}\leqslant
	1+\zeta^+\max\{\|W^+\|_{0,0;\mathbb R},\|W^-\|_{0,0;\mathbb R}\}$.
	The inclusion $\mathcal K^{(\zeta)}\subset\mathcal K(\eps)\Subset\mathcal D$
	yields constants
	$0<\lambda\leqslant\Lambda<\infty$, independent of $\zeta$,
	for which the following estimates hold:
	\begin{equation}\label{Est:uniform}
		\begin{gathered}
			\mathbf e^\top\mathrm D_{\bxi}\mathbf A(S,B,\nabla\phi^{(\zeta)})\mathbf e
			\geqslant\lambda|\mathbf e|^2
			\quad\text{for all }\mathbf e\in\mathbb R^2,\\
			\|\mathrm D_{\bxi}\mathbf A(S,B,\nabla\phi^{(\zeta)})\|_{0,0;\N^\pm}
			\leqslant\Lambda.
		\end{gathered}
	\end{equation}
	
	The bounds in \eqref{Est:uniform} allow us to apply the De Giorgi--Nash estimates before a uniform H\"older bound for the coefficients is available. We estimate the tangential derivative and the pressure near the contact discontinuity, then recover the normal derivative from the algebraic relations. Near the walls, a boundary estimate and a local energy argument give the corresponding gradient bound. The resulting $C^{0,\alpha_0}$ coefficient bounds are the inputs to the Schauder estimates proving \eqref{Est:U-alpha}.
	
	Following \cite{Chen2006}, we derive an elliptic system for
	$w:=\phi_{y_1}$ and $p$. The recovered variables
	$\mathbf U=(\mathbf u,p,\rho)$ satisfy the transformed system
	\eqref{Equ:lang}, which can be written as
	\begin{equation}\label{Equ:matrix}
		\mathbf H(\mathbf U) \mathbf U_{y_1} + \mathbf G(\mathbf U) \mathbf U_{y_2} = 0,
	\end{equation}
	where
	\begin{align}
		\mathbf H(\mathbf U) &=
		\begin{pmatrix}
			-\dfrac{1}{\rho u_1^2} & 0  & 0  & -\dfrac{1}{\rho^2 u_1}\\[1.8ex]
			1-\dfrac{p}{\rho u_1^2} & 0   & \dfrac{1}{\rho u_1} & -\dfrac{p}{\rho^2 u_1} \\[1.8ex]
			0 & 1 & 0 & 0 \\[1.8ex]
			u_1 & u_2   & \dfrac{\gamma}{(\gamma-1)\rho} & -\dfrac{\gamma p}{(\gamma-1)\rho^2}
		\end{pmatrix},\\[2ex]
		\mathbf G(\mathbf U) &=
		\begin{pmatrix}
			\dfrac{u_2}{u_1^2} & -\dfrac{1}{u_1} & 0 & 0 \\[1.8ex]
			\dfrac{p u_2}{u_1^2}  & -\dfrac{p}{u_1}  &  -\dfrac{u_2}{u_1}& 0\\[1.8ex]
			0 & 0 & 1 & 0 \\[1.8ex]
			0 & 0 & 0 & 0
		\end{pmatrix}.
	\end{align}
	
	For subsonic flows, the characteristic equation
	\begin{equation}
		\det(\lambda \mathbf H - \mathbf G) = 0
	\end{equation}
	has four eigenvalues
	\begin{equation}
		\lambda_1 = \lambda_2 = 0,\qquad
		\lambda_{3,4} = \lambda_\pm = \lambda_R \pm i \lambda_I,
	\end{equation}
	where
	\begin{equation}
		\lambda_R = -\frac{c^2 \rho u_2}{c^2 - u_1^2},\qquad
		\lambda_I = \frac{c \rho u_1 \sqrt{c^2 - |\mathbf u|^2}}{c^2 - u_1^2}.
	\end{equation}
	The corresponding left eigenvectors, satisfying
	$\boldsymbol\ell_i(\lambda_i\mathbf H-\mathbf G)=0$, are
	\begin{equation}
		\begin{aligned}
			\boldsymbol{\ell}_1 &= (0,\;0,\;0,\;1),\\
			\boldsymbol{\ell}_2 &= (-p u_1,\; u_1,\; u_2,\; -1),\\
			\boldsymbol{\ell}_{3,4} &= \Bigg(
			\left( \frac{\gamma p^2}{(\gamma-1)\rho u_1} - \frac{p u_1}{\gamma-1} \right) \lambda_{3,4}
			+ \frac{\gamma p^2 u_2}{(\gamma-1)u_1},\;\\
			&\quad-\left( u_1 + \frac{\gamma p}{(\gamma-1)\rho u_1} \right)\lambda_{3,4}
			- \frac{\gamma p u_2}{(\gamma-1)u_1},\;
			\frac{\gamma p}{\gamma-1} - u_2 \lambda_{3,4},\;
			\lambda_{3,4}
			\Bigg).
		\end{aligned}
	\end{equation}
	Multiplying \eqref{Equ:matrix} on the left by $\boldsymbol{\ell}_1$ and $\boldsymbol{\ell}_2$ yields the Bernoulli law and the entropy relation, respectively (cf. \eqref{Equ:SB}). Multiplication by $\boldsymbol{\ell}_3$, followed by separation into real and imaginary parts, gives the elliptic system
	\begin{equation}
		\begin{aligned}
			\mathrm D_R w + e \, \mathrm D_I p &= 0,\\
			\mathrm D_I w - e \, \mathrm D_R p &= 0,
		\end{aligned}
		\label{eq:wp_system}
	\end{equation}
	where the differential operators are defined by
	\begin{equation}
		\mathrm D_R = \partial_{y_1} + \lambda_R \partial_{y_2},\qquad
		\mathrm D_I = \lambda_I \partial_{y_2},
	\end{equation}
	and
	\begin{equation}
		e = \frac{\sqrt{c^2 - |\mathbf u|^2}}{c \rho u_1^2}.
	\end{equation}
	Using \eqref{Def:velocity-Lag} and \eqref{Def:A}, we rewrite
	\eqref{eq:wp_system} as
	\begin{equation}\label{Equ:wp}
		\nabla^\perp_{\y} p = \mathrm D_{\bxi} \mathbf A \cdot \nabla w,
	\end{equation}
	where $\nabla^\perp_{\y}p=(-\partial_{y_2}p,\partial_{y_1}p)$. Taking the divergence gives
	\begin{equation}	\label{Equ:w_second}
		\div_{\y} \big(\mathrm D_{\bxi} \mathbf A \cdot \nabla w\big) = 0.
	\end{equation}
	Write $A:=\mathrm D_{\bxi}\mathbf A$ and
	$J:=\left(\begin{smallmatrix}0&-1\\1&0\end{smallmatrix}\right)$.
	Since $A$ is symmetric, $AJA=(\det A)J$.
	Thus \eqref{Equ:wp} gives
	\[
	\frac{A}{\det A}\nabla p=-J\nabla w,
	\]
	and therefore
	\begin{equation}\label{eq:p_second}
		\div_{\y}\big((\det\mathrm D_{\bxi}\mathbf A)^{-1}
		\mathrm D_{\bxi}\mathbf A\,\nabla p\big)=0.
	\end{equation}
	The continuity of $\phi$ and the weak equation imply $[w]=0$
	and $[p]=0$. Since both functions belong to $C^{1,\alpha}(\overline{\N^+})\cap C^{1,\alpha}(\overline{\N^-})$,
	they belong to $H^1_{\mathrm{loc}}(\N)$. The rotated-gradient
	identities consequently hold in distributions across
	$\Gamma^*$. In particular, \eqref{Equ:w_second} and
	\eqref{eq:p_second} hold weakly in balls intersecting the contact discontinuity $\Gamma^*$.
	If the eigenvalues of $A$ lie in $[\lambda,\Lambda]$, those of
	$A/\det A$ lie in $[\Lambda^{-1},\lambda^{-1}]$.
	
	The interior and boundary De Giorgi--Nash estimates provide
	positive H\"older exponents determined by the uniform
	ellipticity bounds and the wall data. Choose
	$\alpha_0\in(0,\alpha]$ no larger than these exponents. We first bound $\phi$ in $C^{1,\alpha_0}$ on both sides
	of the contact discontinuity.
	
	\begin{lemma}\label{Lm:inner}
		Suppose that $\phi$ satisfies \eqref{Equ:cond}. Then for every $\mathbf y \in \Gamma^*$,
		\begin{equation}\label{Est:mid-beta}
			\|\phi\|_{1,\alpha_0; B^\pm_{3\underline m/4}(\mathbf y)} \leqslant \overline C.
		\end{equation}
	\end{lemma}
	
	\begin{proof}
		Theorem~8.24 in \cite{Gilbarg1997Elliptic}, with the common exponent $\alpha_0$ chosen above, gives
		\begin{equation}
			\|\partial_{y_1}\phi\|_{0,\alpha_0; B_{3\underline m/4}(\mathbf y)}
			= \|w\|_{0,\alpha_0; B_{3\underline m/4}(\mathbf y)}
			\leqslant\overline C \|w\|_{0,0; B_{7\underline m/8}(\mathbf y)},
		\end{equation}
		and
		\begin{equation}
			\|p\|_{0,\alpha_0; B_{3\underline m/4}(\mathbf y)}
			\leqslant\overline C \|p\|_{0,0; B_{7\underline m/8}(\mathbf y)}.
		\end{equation}
		Furthermore, from \eqref{Equ:X} and \eqref{Def:flow-recovery}, we deduce that
		\begin{equation}
			\phi_{y_2} = \sqrt{\frac{1+\phi_{y_1}^2}{2Q(S, B, \rho)}},\qquad
			\rho = \left(\frac{\gamma p}{(\gamma-1)S}\right)^{1/\gamma}.
		\end{equation}
		Since $Q$ is smooth and
		\begin{equation}
			Q(S, B, \rho)=\frac{1+\phi_{y_1}^2}{2\phi_{y_2}^2}\geqslant\frac{\eps^2}{2},
		\end{equation}
		the formulas recover $\rho$ from $(S,p)$ and $\phi_{y_2}$ from $(S,B,p,w)$ by smooth functions on the compact range fixed by \eqref{Equ:cond}.
		The preceding estimates for $w,p$, together with the given
		regularity of $S,B$, therefore bound $\phi_{y_2}$ in
		$C^{0,\alpha_0}$ separately in the two half-balls.
		Combining these bounds with the already established bound
		for $\phi$ proves \eqref{Est:mid-beta}.
	\end{proof}

	Near the nozzle walls, the prescribed trace gives a boundary H\"older estimate for $w=\phi_{y_1}$. A local energy estimate for $w$, followed by the original equation and Morrey's estimate, then controls $\phi_{y_2}$.
	
	\begin{lemma}\label{Lm:boundary}
		Suppose that $\phi$ satisfies \eqref{Equ:cond}.
		Then, for every $q\in\mathbb R$, we have
		\begin{equation}\label{Est:boundary-holder}
			\|\phi\|_{1,\alpha_0;D_{m^\pm-\underline m/8}(\y_q^\pm)}
			\leqslant\overline C.
		\end{equation}
	\end{lemma}
	
	\begin{proof}
		Fix $q\in\mathbb R$. We first prove the estimate near the
		upper wall $\Gamma^+$.
		Since $w=\phi_{y_1}$ satisfies \eqref{Equ:w_second} with
		$w=\zeta(W^+)'$ on $\Gamma^+$, Theorem~8.29 in
		\cite{Gilbarg1997Elliptic} gives
		\begin{equation}
			\|\partial_{y_1}\phi\|_{0,\alpha_0; D_{m^+-\underline m/16}(\y_q^+)}
			\leqslant
			C\left(
			\|\partial_{y_1}\phi\|_{0,0; D_{m^+-\underline m/32}(\y_q^+)}
			+ |\zeta| \, \|W^+\|_{1,\alpha;\mathbb R}
			\right).
		\end{equation}
		Since $0\leqslant\zeta<\zeta^+$ and $w$ is uniformly bounded
		under \eqref{Equ:cond}, this gives
		\begin{equation}
			\|\partial_{y_1}\phi\|_{0,\alpha_0; D_{m^+-\underline m/16}(\y_q^+)}  \leqslant\overline C.
		\end{equation}
		
		We next estimate $\partial_{y_2}\phi$. Fix $\mathbf z \in D_{m^+-\underline m/8}(\y_q^+)$ and choose $0<r\leqslant\min\{1,\underline m/64\}$. Let $\eta \in C_c^\infty(B_{2r}(\mathbf z))$ be a cutoff function satisfying $\eta = 1$ on $B_r(\mathbf z)$ and $|\nabla \eta| < 2/r$. Define
		\begin{equation}
			\overline w(\y) =
			\begin{cases}
				\zeta (W^+)'(y_1), & \text{if } B_{2r}(\mathbf z) \cap \Gamma^+ \neq \varnothing, \\[4pt]
				w(\mathbf z), & \text{otherwise},
			\end{cases}
		\end{equation}
		and set $v = \eta^2 (w - \overline w)$. Then $v \in H^1_0\big(D_{m^+-\underline m/16}(\y_q^+)\big)$, and from \eqref{Equ:w_second} we obtain
		\begin{equation}
			\begin{aligned}
				&\int_{D_{m^+-\underline m/16}(\y_q^+)}
				\eta^2 A^i_{\xi_j} w_{y_j} w_{y_i} \\
				&=
				\int_{D_{m^+-\underline m/16}(\y_q^+)}
				\eta^2 A^i_{\xi_j} w_{y_j} \partial_{y_i} \overline w
				-
				2 \int_{D_{m^+-\underline m/16}(\y_q^+)}
				\eta \, (w - \overline w) \, A^i_{\xi_j} w_{y_j} \partial_{y_i}\eta.
			\end{aligned}
		\end{equation}
		Using \eqref{Est:uniform} and Young's inequality to absorb
		the terms containing $\eta\nabla w$, we obtain
		\begin{equation}
			\int_{B_r(\mathbf z) \cap \N^+} |\nabla w|^2
			\leqslant
			\overline C \left(
			r^2 + \sup_{B_{2r}(\mathbf z)\cap \N^+} (w - \overline w)^2
			\right)
			\leqslant\overline C r^{2\alpha_0}.
		\end{equation}
		For balls meeting the wall, compare $w$ with its prescribed
		trace at a wall point within distance $2r$ of $\mathbf z$;
		for the remaining balls, compare it with $w(\mathbf z)$.
		The H\"older bound for $w$ then gives the last inequality.
		Consequently,
		\begin{equation}\label{Est:int}
			\int_{B_r(\mathbf z) \cap \N^+} |\partial_{y_i}\partial_{y_j}\phi|^2
			\leqslant\overline C r^{2\alpha_0}
		\end{equation}
		for $(i,j)\ne(2,2)$.
		
		To estimate $\phi_{y_2y_2}$, we rewrite the elliptic equation \eqref{Equ:elliptic} as
		\begin{equation}\label{Equ:y22-phi}
			A^2_{\xi_2} \partial_{y_2}\partial_{y_2}\phi
			=
			-\Big(
			A^1_{\xi_1}\partial_{y_1}\partial_{y_1}\phi
			+ 2A^2_{\xi_1}\partial_{y_1}\partial_{y_2}\phi
			+ A^2_S S'
			+ A^2_B B'
			\Big).
		\end{equation}
		The terms containing $\phi_{y_1y_1}$ and $\phi_{y_1y_2}$ are controlled by \eqref{Est:int}. The remaining terms are uniformly bounded by the prescribed $C^{1,\alpha}$ norms of $S,B$ and the compact inclusion $\mathcal K(\eps)\Subset\mathcal D$, so their squared integrals are bounded by $\overline C r^2$. Since $r\leqslant1$ and $A^2_{\xi_2}\geqslant\lambda>0$, \eqref{Est:int} also holds for $(i,j)=(2,2)$.
		
		For balls meeting the nozzle wall, extend $\phi_{y_2}$ evenly
		across the wall; the preceding local integral bounds remain
		valid, up to a fixed constant, for this extension.
		Then, an application of Morrey's estimate (see, e.g., Theorem~7.19 in \cite{Gilbarg1997Elliptic}) yields the desired H\"older estimate:
		\begin{equation}
			\|\partial_{y_2}\phi\|_{0,\alpha_0; D_{m^+-\underline m/8}(\y_q^+)} \leqslant\overline C.
		\end{equation}
		The same argument applies near the lower wall $\Gamma^-$.
		Adding the two estimates proves \eqref{Est:boundary-holder}.
	\end{proof}
	
	We now prove Proposition~\ref{Pr:asymptotic-zeta}.
	
	\begin{proof}[Proof of Proposition~\ref{Pr:asymptotic-zeta}]
		We use the three strips defined in the proof of
		Proposition~\ref{Pr:app-estimate}, with $s=\infty$. Thus
		\begin{equation}\label{Est:contain}
			\begin{aligned}
				\N^+&=\mathcal U_\infty^+
				\cup(\mathcal U_\infty^*\cap\N^+),\qquad
				\N^-&=\mathcal U_\infty^-
				\cup(\mathcal U_\infty^*\cap\N^-).
			\end{aligned}
		\end{equation}
		Applying Lemmas~\ref{Lm:inner} and~\ref{Lm:boundary} for all
		$q\in\mathbb R$ gives
		\[
		\|\phi\|_{1,\alpha_0;\mathcal U_\infty^\pm}
		+\|\phi\|_{1,\alpha_0;\mathcal U_\infty^*\cap\N^\pm}
		\leqslant\overline C.
		\]
		The half-balls in Lemmas~\ref{Lm:inner} and~\ref{Lm:boundary} have the same radii as those used in Step~1 of Proposition~\ref{Pr:app-estimate}. Translating them and using the overlaps of width $\underline m/4$ therefore gives
		\begin{equation}\label{Est:varphi-beta}
			\|\phi\|_{1,\alpha_0;\N^\pm}\leqslant\overline C.
		\end{equation}
		The smoothness of $\mathbf A$ on $\mathcal K(\eps)$,
		together with \eqref{Est:varphi-beta} and the prescribed
		regularity of $S$ and $B$, gives
		\[
		\max_{i,j\in\{1,2\}}
		\|A^i_{\xi_j}(S,B,\nabla\phi)\|_{0,\alpha_0;\N^\pm}
		\leqslant\overline C.
		\]
		We claim that
		\begin{equation}\label{Est:beta-2}
			\|\phi\|_{2,\alpha_0;\N^\pm}\leqslant\overline C.
		\end{equation}
		Here we differentiate the nonlinear flux equation
		\eqref{Equ:elliptic}. Since $S$ and $B$ are independent
		of $y_1$, the equation for $w=\phi_{y_1}$ is the homogeneous
		equation \eqref{Equ:w_second}, with
		$w=\zeta(W^\pm)'$ on $\Gamma^\pm$. The following Schauder
		estimates use the $C^{0,\alpha_0}$ coefficient bound just
		obtained.
		For $q\in\mathbb R$, adding the upper-wall and lower-wall
		boundary Schauder estimates gives
		\[
		\begin{aligned}
			&\|w\|_{1,\alpha_0;D_{m^\pm-\underline m/8}(\y_q^\pm)}\leqslant\overline C\left(
			\|w\|_{0,0;D_{m^\pm}(\y_q^\pm)}
			+\zeta^+\|(W^\pm)'\|_{1,\alpha_0;\mathbb R}\right).
		\end{aligned}
		\]
		The Schauder estimate in \cite{Chen2024Elliptic} gives
		\[
		\|w\|_{1,\alpha_0;B_{3\underline m/4}^\pm(\y_q^*)}
		\leqslant C\|w\|_{0,0;B_{\underline m}(\y_q^*)}.
		\]
		Applying these local estimates for all $q\in\mathbb R$ yields
		\[
		\|w\|_{1,\alpha_0;\mathcal U_\infty^\pm}
		+\|w\|_{1,\alpha_0;\mathcal U_\infty^*\cap\N^\pm}
		\leqslant\overline C.
		\]
		Combining the three strip estimates by \eqref{Est:contain}
		and their fixed overlap, we obtain
		\begin{equation}\label{Est:psi-regularity}
			\|\partial_{y_1}\phi\|_{1,\alpha_0;\N^\pm}
			=\|w\|_{1,\alpha_0;\N^\pm}\leqslant\overline C.
		\end{equation}
		Equation~\eqref{Equ:y22-phi}, together with
		$A^2_{\xi_2}\geqslant\lambda$, then gives
		\begin{equation}\label{Est:y2-regularity}
			\|\partial_{y_2}^2\phi\|_{0,\alpha_0;\N^\pm}
			\leqslant\overline C.
		\end{equation}
		This proves \eqref{Est:beta-2}.
		
		The bound \eqref{Est:beta-2} makes $\nabla\phi$ uniformly
		Lipschitz in each of the two convex strips. For pairs of points
		at distance at most one, this controls the H\"older quotient
		with exponent $\alpha$; for the remaining pairs, the
		$L^\infty$ bound suffices. Thus
		\[
		\|\phi\|_{1,\alpha;\N^\pm}\leqslant\overline C,
		\]
		and the prescribed regularity of $S,B$ gives a uniform
		$C^{0,\alpha}$ bound for the coefficients in both regions.
		Applying the upper-wall, lower-wall, and contact
		discontinuity estimates to $w$ with exponent $\alpha$, and
		combining them on the same three strips, gives
		\[
		\|w\|_{1,\alpha;\N^\pm}\leqslant\overline C.
		\]
		Recovering $\phi_{y_2y_2}$ from \eqref{Equ:y22-phi} proves
		\eqref{Est:U-alpha}. The uniform-radius continuation argument
		given above now yields \eqref{Est:del-2-phi}, completing the
		proof.
	\end{proof}
	
	The local existence result and Proposition~\ref{Pr:asymptotic-zeta}
	prove Theorem~\ref{Thm:elliptic} at the right endpoint. The same
	argument applies at a finite left endpoint.

	\section{Return to the physical coordinates}\label{Sec:main}
	
	Let $\phi$ be the solution of Problem~\ref{Problem:elliptic} given by Theorem~\ref{Thm:elliptic}, and recover $\U$ from \eqref{Def:flow-recovery}. We verify the Euler equations, the boundary conditions, and the upstream limit in physical coordinates. The identity
	$\partial_{y_1}\phi_{y_2}=\partial_{y_2}\phi_{y_1}$ gives the first
	equation in \eqref{Equ:lang}, and the scalar equation gives
	$(u_2)_{y_1}+p_{y_2}=0$. Bernoulli's law and $S_{y_1}=0$ imply
	\[
	u_1(u_1)_{y_1}+u_2(u_2)_{y_1}+\frac{p_{y_1}}\rho=0.
	\]
	Using the first and third equations already verified, we obtain
	\[
	\begin{aligned}
		\partial_{y_1}\left(u_1+\frac{p}{\rho u_1}\right)
		-\partial_{y_2}\left(\frac{pu_2}{u_1}\right)
		&= (u_1)_{y_1}+\frac{p_{y_1}}{\rho u_1}
		-\frac{u_2}{u_1}p_{y_2}=0.
	\end{aligned}
	\]
	Thus the transformed Euler system \eqref{Equ:lang} holds in both $\N^+$ and $\N^-$. On $\Gamma^*$, continuity of $\phi$ gives the common
	tangential trace $[\phi_{y_1}]=0$, and the weak scalar equation
	gives $[A^2]=[p]=0$. These are exactly
	\eqref{Cond:RH-y}.
	
	We return to the physical variables by defining the inverse
	Lagrangian transformation
	\[
	\mathcal S:\N\to\N_\zeta,\qquad
	\mathcal S(y_1,y_2)=\big(y_1,\phi(y_1,y_2)\big).
	\]
	Since $\phi_{y_2}$ has a positive lower bound and $\phi$ has
	the prescribed wall values, $\mathcal S$ is a $C^{2,\alpha}$
	diffeomorphism from $\N^+$ onto $\N_\zeta^+$ and from
	$\N^-$ onto $\N_\zeta^-$. The two restrictions have the same trace on $\Gamma^*$. Using the same notation for the physical variables, set
	\[
	\mathbf U(\x):=\mathbf U(\mathcal S^{-1}(\x)),\qquad
	\x\in\N_\zeta^+\cup\N_\zeta^-.
	\]
	The equations in \eqref{Equ:lang} and the jump conditions
	\eqref{Cond:RH-y} show that $\mathbf U$ is a weak solution of
	\eqref{Equ:euler}. It satisfies \eqref{Cond:slip} on the walls
	$\Gamma_\zeta^+$ and $\Gamma_\zeta^-$ and \eqref{Cond:RH-intro} on the contact
	discontinuity $\Gamma_\zeta^*$.
	
	We next verify the remaining assertions of Theorem~\ref{Thm:main}. Since $\mathcal K[\phi]\Subset\mathcal D$,
	Lemma~\ref{Lm:K-equiv} gives condition
	\ref{Cond:p4}. The upstream gradient convergence, integrated
	from the lower wall, also yields
	$\phi(y_1,\cdot)\to\phi_\ell$ uniformly. Since
	$\phi_{y_2}$ has a positive lower bound, the inverse transverse
	maps converge on each upstream region. Composing the recovered variables with these maps gives the upstream condition in Problem~\ref{Problem:physical}. Uniqueness follows from Theorem~\ref{Thm:elliptic} and the Euler--Lagrange correspondence established in Section~\ref{Sec:reformulation}.
	
	To prove \eqref{Est:blowup}, suppose that the asserted limit
	fails at a finite endpoint $\zeta_*$. Then there are a sequence
	$\zeta_n\in(\zeta^-,\zeta^+)$ converging to $\zeta_*$ and
	$\varepsilon>0$ such that
	\[
	d_{\rm s}(\zeta_n)\geqslant\varepsilon,
	\qquad d_{\rm m}(\zeta_n)\geqslant\varepsilon
	\quad\text{for all }n.
	\]
	The fixed positive functions $S$ and $B$ give common upper and
	lower bounds for the density on the subsonic branch and a
	common upper bound for the speed. The lower bound for
	$d_{\rm m}(\zeta_n)$ bounds $\phi_{y_2}=1/(\rho u_1)$ and
	$\phi_{y_1}=u_2/u_1$. The lower bound for $d_{\rm s}(\zeta_n)$
	keeps the density away from the critical density. Thus the sets
	$\mathcal K^{(\zeta_n)}$ are contained in a common compact subset
	of $\mathcal D$, and \eqref{Equ:cond} holds with a common
	smaller constant. This contradicts the limit in
	Proposition~\ref{Pr:asymptotic-zeta}, or its counterpart at the
	left endpoint, and proves \eqref{Est:blowup}.
	
	Finally, the image of $\Gamma^*$ under $\mathcal S$ is the
	contact discontinuity
	\[
	\Gamma^*_\zeta = \left\{(x_1,x_2)\in \mathbb R^2 : x_2 = \phi(x_1,0)\right\}.
	\]
	Thus $\omega^*(x_1)=\phi(x_1,0)\in C^{2,\alpha}(\mathbb R)$,
	which proves the remaining assertion of Theorem~\ref{Thm:main}.
	
	\section{Higher regularity of the solutions}\label{Sec:regularity}

	In this section we prove Theorem~\ref{Thm:regularity}. We first estimate tangential derivatives by the wall and contact discontinuity Schauder estimates, then recover the remaining derivatives from the equation separately in $\N^+$ and $\N^-$.
	
	\begin{proposition}\label{Pr:regularity-varphi}
		Let $k\geqslant1$ be an integer. Suppose that
		\begin{equation}\label{Cond:reg}
			\begin{gathered}
				S,B\in C^{k,\alpha}([-m^-,0])\cap C^{k,\alpha}([0,m^+]),
				W^\pm\in C^{k+1,\alpha}(\mathbb R).
			\end{gathered}
		\end{equation}
		Then the solution $\phi$ of Problem~\ref{Problem:elliptic} satisfies
		\begin{equation}\label{Est:varphi-regularity}
			\|\phi\|_{k+1,\alpha;\N^\pm}\leqslant C,
		\end{equation}
		where $C$ depends on $k$, the data, the $C^{2,\alpha}$ bound for
		$\phi$, and a compact subset of $\mathcal D$ containing
		$\mathcal K[\phi]$.
	\end{proposition}
	
	\begin{proof}
		The case $k=1$ is part of the construction in Section~\ref{Sec:proof-main}. Suppose that the estimate holds for $k=n-1$, where $n\geqslant2$, and that the data satisfy \eqref{Cond:reg} with $k=n$.
		Put $a_{ij}:=A^i_{\xi_j}(S,B,\nabla\phi)$ and
		$w:=\phi_{y_1}$. Since $S$ and $B$ are independent of $y_1$,
		\[
		\partial_{y_i}(a_{ij}\partial_{y_j}w)=0
		\]
		holds weakly across $\Gamma^*$. Differentiate this equation
		$n-1$ times in the tangential direction. With
		$u_n:=\partial_{y_1}^n\phi$, we obtain
		\begin{equation}\label{Equ:k-y1}
			\partial_{y_i}(a_{ij}\partial_{y_j}u_n)
			=\div\mathbf F\quad\text{in }\N,
		\end{equation}
		where $\mathbf F=(F_1,F_2)$ and
		\[
		F_i=-\sum_{m=1}^{n-1}\binom{n-1}{m}
		(\partial_{y_1}^m a_{ij})
		\partial_{y_j}\partial_{y_1}^{n-m}\phi,
		\qquad i=1,2.
		\]
		Each factor involves derivatives of $\phi$ of order at most
		$n$. The induction hypothesis therefore bounds $\mathbf F$
		in $C^{0,\alpha}$ on both $\N^+$ and $\N^-$. Tangential difference
		quotients justify \eqref{Equ:k-y1} before the additional
		regularity is known. In particular, the conditions across
		the contact discontinuity are
		\[
		[u_n]=0,\qquad [a_{2j}\partial_{y_j}u_n-F_2]=0,
		\]
		and $u_n=\zeta\partial_{y_1}^nW^\pm$ on $\Gamma^\pm$.
		The upper-wall and lower-wall boundary estimates and the
		Schauder estimate in \cite{Chen2024Elliptic}, as in Step~2 of
		Proposition~\ref{Pr:app-estimate}, give
		\begin{equation}\label{Est:n-phi}
			\|\partial_{y_1}^n\phi\|_{1,\alpha;\N^\pm}
			\leqslant C\left(
			\|u_n\|_{0,0;\N^\pm}+\|\mathbf F\|_{0,\alpha;\N^\pm}
			+|\zeta|\|W^\pm\|_{n+1,\alpha;\mathbb R}\right)
			\leqslant C.
		\end{equation}
		
		We recover the remaining derivatives separately in $\N^+$
		and $\N^-$ from \eqref{Equ:y22-phi}, taking the derivatives
		distributionally until their regularity has been established. For $0\leqslant m\leqslant n+1$, we claim
		\begin{equation}\label{Est:arb}
			\|\partial_{y_1}^{n+1-m}\partial_{y_2}^m\phi\|_{0,\alpha;\N^\pm}
			\leqslant C.
		\end{equation}
		The cases $m=0,1$ follow from \eqref{Est:n-phi}. Suppose all
		cases up to $m=m_0$ have been proved, where
		$1\leqslant m_0\leqslant n$. Apply
		$\partial_{y_1}^{n-m_0}\partial_{y_2}^{m_0-1}$ to
		\eqref{Equ:y22-phi} in each region. The result has the form
		\[
		A^2_{\xi_2}\,
		\partial_{y_1}^{n-m_0}\partial_{y_2}^{m_0+1}\phi=f.
		\]
		Terms of total order $n+1$ in $f$ contain at most $m_0$
		normal derivatives and have already been estimated. All other
		derivatives of $\phi$ in $f$ have total order at most $n$;
		these may contain more than $m_0$ normal derivatives, but are
		controlled by the outer induction hypothesis. Derivatives of
		$S$ and $B$ have order at most $n$ and are controlled by the
		data. Thus $f$ has a bounded $C^{0,\alpha}$ norm on each region.
		Since $A^2_{\xi_2}\geqslant\lambda>0$, division gives
		\eqref{Est:arb} for $m=m_0+1$. This proves the induction step
		and hence \eqref{Est:varphi-regularity}.
	\end{proof}
	
	\begin{proof}[Proof of Theorem~\ref{Thm:regularity}]
		The upstream change of coordinate $\vartheta_\ell$ and its inverse are piecewise $C^{k+1,\alpha}$, since $\rho_\ell u_\ell$ is piecewise $C^{k,\alpha}$ and has a positive lower bound. Thus the transformed upstream data are piecewise $C^{k,\alpha}$, and \eqref{Equ:SB} gives \eqref{Cond:reg}. Proposition~\ref{Pr:regularity-varphi} therefore yields
		\[
		\|\phi\|_{k+1,\alpha;\N^\pm} \leqslant C.
		\]
		By \eqref{Def:flow-recovery}, the flow variables $\mathbf U=(\mathbf u,p,\rho)$ are smooth functions of $S$, $B$, and $\nabla\phi$ on the fixed compact subset of $\mathcal D$ containing $\mathcal K[\phi]$. In each of $\N^+$ and $\N^-$, we have $S,B\in C^{k,\alpha}$
		and $\phi\in C^{k+1,\alpha}$. Hence
		\[
		\mathbf U \in C^{k,\alpha}(\overline{\N^+})\cap C^{k,\alpha}(\overline{\N^-}).
		\]
		Returning to the physical coordinates via the inverse Euler--Lagrange transformation $\mathcal S$ (whose restrictions to $\N^+$ and $\N^-$ are $C^{k+1,\alpha}$
		diffeomorphisms by the regularity of $\phi$), we conclude that
		$\mathbf U$ belongs to
		$C^{k,\alpha}(\overline{\N_\zeta^+})\cap
		C^{k,\alpha}(\overline{\N_\zeta^-})$. Moreover, $\omega^*(x_1)=\phi(x_1,0)$ belongs to $C^{k+1,\alpha}(\mathbb R)$. This completes the proof of Theorem~\ref{Thm:regularity}.
	\end{proof}

	\section{Downstream asymptotic behavior}\label{Sec:outlet}
	We now study the downstream limit and convergence rates. Since $y_1=x_1$, we work in Lagrangian coordinates and assume that the nozzle walls approach limiting positions: there exist constants $a^+,a^-\in\mathbb R$ such that
	\begin{equation}\label{Cond:r}
		W^\pm(x_1)\to a^\pm \quad \text{as } x_1\to+\infty.
	\end{equation}
	
	For $-\infty\leqslant L<R\leqslant +\infty$, we introduce the truncated domain
	\begin{equation*}
		\N_{L,R}:=\N\cap\{L<y_1<R\},
	\end{equation*}
	and set
	\begin{equation*}
		\N^\pm_{L,R}:=\N^\pm\cap\{L<y_1<R\}.
	\end{equation*}
	
	\begin{theorem}\label{Thm:asymptotic}
		Suppose that $W^+$ and $W^-$ satisfy \eqref{Cond:r}, and let $\phi$ be the solution of Problem~\ref{Problem:elliptic}. Then there exists a function
		\[
		\phi_r \in C^{2,\alpha}\big([-m^-,0]\big) \cap C^{2,\alpha}\big([0,m^+]\big) \cap C\big([-m^-,m^+]\big)
		\]
		such that
		\begin{equation}\label{Est:r}
			\lim_{y_1\to+\infty}
			\big\|\nabla\phi(y_1,\cdot)-\nabla\phi_r\big\|_{L^\infty([-m^-,m^+])}
			=0.
		\end{equation}
		Here $\phi_r$ is extended to $\N$ by independence of $y_1$.
		
		Moreover, if there exist constants $\beta>0$ and $C_0>0$ such that
		\begin{equation}\label{Cond:decay}
			|W^+(y_1)-a^+|+|W^-(y_1)-a^-|
			\leqslant\frac{C_0}{(1+y_1)^\beta},\qquad y_1>0,
		\end{equation}
		then there exists a constant $C>0$ such that
		\begin{equation}\label{Est:beta}
			\|\phi-\phi_r\|_{L^\infty(\N_{Y,\infty})}
			\leqslant\frac{C}{(1+Y)^\beta},\qquad Y>0.
		\end{equation}
		
		The constant $C$ is independent of $Y$.
	\end{theorem}
	
	\begin{theorem}\label{Thm:decay-C2}
		Let $\phi$ and $\phi_r$ be as in Theorem~\ref{Thm:asymptotic}.
		Suppose that there exist constants $\beta>0$ and $C_1>0$ such
		that
		\begin{equation}\label{Cond:decay-C2}
			\|W^\pm-a^\pm\|_{2,\alpha;(Y,\infty)}
			\leqslant\frac{C_1}{(1+Y)^\beta},\qquad Y>0.
		\end{equation}
		Then there exists a constant $C>0$, independent of $Y$, such that
		\begin{equation}\label{Est:beta-C2}
			\|\phi-\phi_r\|_{2,\alpha;\N^\pm_{Y,\infty}}
			\leqslant\frac{C}{(1+Y)^\beta},\qquad Y>0.
		\end{equation}
		The norms in \eqref{Cond:decay-C2} are the ordinary
		H\"older norms on the indicated tails.
	\end{theorem}
	
	We first prove Theorem~\ref{Thm:asymptotic} in three steps.
	The proof of Theorem~\ref{Thm:decay-C2} is given at the end of
	this section.
	
	\medskip\noindent\textbf{Step 1. Existence of the limiting profile $\phi_r$.} Consider the boundary value problem
	\begin{equation}\label{Equ:inf}
		\begin{gathered}
			\Lop(\phi_r)=0\quad\text{in }\N, \qquad
			\phi_r=\pm1 + \zeta a^\pm
			\quad\text{on }\Gamma^\pm.
		\end{gathered}
	\end{equation}
	
	\begin{proposition}\label{Pr:existence-r}
		Suppose $\zeta\in(\zeta^-,\zeta^+)$. Equation~\eqref{Equ:inf} admits
		a unique solution
		\[
		\phi_r\in C^{2,\alpha}(\overline{\N^+})\cap C^{2,\alpha}(\overline{\N^-})\cap C(\overline{\N}),
		\]
		with $\mathcal K[\phi_r]\Subset\mathcal D$.
		Moreover, $\phi_r$ is independent of $y_1$, i.e. $\phi_r=\phi_r(y_2)$.
	\end{proposition}
	\begin{proof}
		For $n\in\mathbb N$, define
		\[
		\phi^n(y_1,y_2):=\phi(y_1+n,y_2).
		\]
		The translates have uniformly bounded $C^{2,\alpha}$ norms
		in both $\N^+$ and $\N^-$.
		Since $S$ and $B$ depend only on $y_2$,
		$\mathcal K[\phi^n]=\mathcal K[\phi]\Subset\mathcal D$ for every $n$.
		Arzel\`a--Ascoli on bounded substrips and a diagonal extraction
		therefore give a subsequence converging in
		\[
		C^{2,\alpha/2}_{\mathrm{loc}}(\overline{\N^+})\cap C^{2,\alpha/2}_{\mathrm{loc}}(\overline{\N^-})
		\cap C_{\mathrm{loc}}(\overline\N),
		\]
		to a function $\phi_r$. The global $C^{2,\alpha}$ bound is
		preserved in the limit, and
		$\mathcal K[\phi_r]\subset\mathcal K[\phi]\Subset\mathcal D$.
		Passing to the weak equation and to the wall traces using
		\eqref{Cond:r} shows that $\phi_r$ solves \eqref{Equ:inf}.
		Uniqueness follows from the same argument as in
		Proposition~\ref{Pr:existence-varphi}.
		
		For any $h\in\mathbb R$, the translate
		$\phi_r^h(y_1,y_2):=\phi_r(y_1+h,y_2)$ solves the same problem,
		since $S$, $B$, and the wall data are independent of $y_1$.
		Moreover, $\mathcal K[\phi_r^h]=\mathcal K[\phi_r]\Subset\mathcal D$.
		Uniqueness gives $\phi_r^h=\phi_r$. Since $h$ is arbitrary,
		$\phi_r$ is independent of $y_1$.
	\end{proof}
	
	\medskip\noindent\textbf{Step 2. Convergence to the limiting profile.}
	The uniqueness of the limiting profile also gives convergence of the full family of downstream translates.
	
	\begin{proposition}\label{Pr:asymptotic-r}
		Let $\phi_r$ be the profile constructed in Proposition~\ref{Pr:existence-r}. Then
		\[
		\lim_{y_1\to +\infty}
		\big\|\nabla\phi(y_1,\cdot)-\nabla\phi_r\big\|_{L^\infty([-m^-,m^+])}=0.
		\]
	\end{proposition}
	\begin{proof}
		Let $q_n\to+\infty$. The translated functions $\phi(q_n+\cdot,\cdot)$ have the same $C^{2,\alpha}$ bound in both regions, and their sets $\mathcal K[\phi(q_n+\cdot,\cdot)]$ all equal $\mathcal K[\phi]\Subset\mathcal D$.
		Every subsequence has a further subsequence converging in
		$C^{2}_{\mathrm{loc}}(\overline{\N^+})\cap C^{2}_{\mathrm{loc}}(\overline{\N^-})$ to a solution of
		\eqref{Equ:inf}. The wall traces converge to the prescribed
		constants by \eqref{Cond:r}. Proposition~\ref{Pr:existence-r}
		identifies every such limit with $\phi_r$.
		In particular, the gradients on the compact cross-section
		$y_1=0$ converge uniformly on each of the intervals $[-m^-,0]$ and $[0,m^+]$.
		Since the sequence $q_n$ was arbitrary, this proves
		\eqref{Est:r}.
	\end{proof}

	\medskip\noindent\textbf{Step 3. Decay rate estimates.}
	To obtain the rate in \eqref{Est:beta}, assume \eqref{Cond:decay} and set
	\[
	\psi(y_1,y_2):=\phi(y_1,y_2)-\phi_r(y_2).
	\]
	Then $\psi$ satisfies
	\begin{equation}\label{Equ:psi}
		\begin{gathered}
			\partial_{y_i}\big(a_{ij}\partial_{y_j}\psi\big)=0\quad\text{in }\N,\qquad
			\psi=\zeta(W^\pm-a^\pm)\quad\text{on }\Gamma^\pm,
		\end{gathered}
	\end{equation}
	where
	\[
	a_{ij}:=\int_0^1 A^i_{\xi_j}
	\bigl(S,B,(1-t)\nabla\phi_r+t\nabla\phi\bigr)\dif t.
	\]
	By Proposition~\ref{Pr:existence-r},
	$\mathcal K[\phi_r]\subset\mathcal K[\phi]\Subset\mathcal D$.
	The convexity argument in the proof of
	Proposition~\ref{Pr:existence-varphi} places all arguments in
	this integral in a fixed compact subset of $\mathcal D$.
	Since $\phi$ and $\phi_r$ have bounded $C^{2,\alpha}$ norms
	in both regions, the coefficients are symmetric and uniformly
	elliptic, and
	\[
	\max_{i,j\in\{1,2\}}\|a_{ij}\|_{1,\alpha;\N^\pm}\leqslant C.
	\]
	These bounds depend on the fixed solution and the prescribed
	data, and are independent of the truncation parameters.
	
	We impose decaying data at the artificial outlet and compare the resulting solutions with a barrier independent of the truncation length. Consider
	\begin{equation}\label{Equ:infa}
		\partial_{y_i}\big(a_{ij}\partial_{y_j}\psi^R\big)=0 \quad\text{in } \N_{0,R},
		\qquad
		\psi^R=g \quad\text{on } \partial\N_{0,R},
	\end{equation}
	with boundary data
	\[
	g(y_1,y_2):=\zeta\mu(y_1)\underline{\psi}(y_1,y_2)+\big(1-\mu(y_1)\big)\psi(y_1,y_2),
	\]
	where $\mu\in C^\infty(\mathbb R)$ satisfies $0\leqslant\mu\leqslant1$,
	$\mu=0$ for $y_1\leqslant1$, and $\mu=1$ for $y_1\geqslant2$, and
	\begin{equation}\label{Def:g-r}
		\underline{\psi}(y_1,y_2)
		:=\chi(y_2)\big(W^-(y_1)-a^-\big)
		+\big(1-\chi(y_2)\big)\big(W^+(y_1)-a^+\big).
	\end{equation}
	Here $\chi$ is the cutoff in \eqref{Def:g}. The boundary data satisfy
	\[
	\begin{gathered}
		g=\psi\quad\text{on }\{y_1=0\},\qquad
		g=\zeta(W^\pm-a^\pm)\quad\text{on }\Gamma^\pm.
	\end{gathered}
	\]
	\begin{proposition}\label{Pr:decay-estimate}
		Assume that \eqref{Cond:decay} holds. Then for every $R>0$, problem \eqref{Equ:infa} admits a unique solution $\psi^R\in H^1(\N_{0,R})$. Moreover, for any $Y$ with $0<Y<R$,
		\[
		\|\psi^R\|_{0,0;\N_{Y,R}}\leqslant \frac{C}{(1+Y)^\beta},
		\]
		where $C$ is independent of $R$.
	\end{proposition}
	\begin{proof}
		Existence and uniqueness follow from the standard theory for linear elliptic equations (see, e.g., Theorems 8.1 and 8.3 in \cite{Gilbarg1997Elliptic}). To derive the decay estimate, introduce the barrier function
		\begin{equation}
			v(y_1,y_2):=C_3\frac{e^{C_1\overline m}-e^{C_1|y_2|}+1}{(C_2+y_1)^\beta},
		\end{equation}
		where the positive constants $C_1$, $C_2$, and $C_3$ are to be determined below.
		
		For $\y\in(\N^+\cup\N^-)\cap\{y_1>0\}$, a direct computation gives
		\begin{equation*}
			\begin{aligned}
				\mathcal L_{\phi,\phi_r} v
				&:=\partial_{y_i}(a_{ij}\partial_{y_j}v) \\
				&=
				\Bigg[
				- a_{22}\frac{C_1^2 e^{C_1|y_2|}}{e^{C_1\overline m}-e^{C_1|y_2|}+1}
				\Bigg(1-\frac{2\beta a_{12}\sgn(y_2)}{a_{22}C_1(C_2+y_1)}
				+\frac{\partial_{y_i}a_{i2}\sgn(y_2)}{a_{22}C_1}\Bigg) \\
				&\quad+
				\Bigg(a_{11}\frac{\beta(\beta+1)}{(C_2+y_1)^2}
				-\frac{\beta\,\partial_{y_i}a_{i1}}{C_2+y_1}\Bigg)
				\Bigg] v,
			\end{aligned}
		\end{equation*}
		Choose $C_1>0$ sufficiently large that, for every $C_2\geqslant1$ and $y_1>0$,
		\[
		\frac{2\beta|a_{12}|}{a_{22}C_1(C_2+y_1)}
		+\frac{|\partial_{y_i}a_{i2}|}{a_{22}C_1}
		\leqslant\frac12,
		\]
		so the first parenthesis is at least $1/2$ in both regions. With $C_1$ fixed, choose $C_2$ large enough that
		\[
		\left|
		a_{11}\frac{\beta(\beta+1)}{(C_2+y_1)^2}
		-\frac{\beta\,\partial_{y_i}a_{i1}}{C_2+y_1}
		\right|
		\leqslant \frac14 a_{22}\frac{C_1^2}{e^{C_1\overline m}},
		\qquad y_1>0.
		\]
		This follows from the uniform bounds for the coefficients and their first derivatives. Hence
		\[
		\mathcal L_{\phi,\phi_r} v
		\leqslant
		-\frac{\lambda C_1^2}{4e^{C_1\overline m}}\,v
		<0
		\qquad\text{in }(\N^+\cup\N^-)\cap\{y_1>0\}.
		\]
		
		On the contact discontinuity $\Gamma^*$, the jump of the conormal derivative satisfies
		\begin{equation*}
			\bigl[a_{2j}\partial_{y_j}v\bigr]
			=
			-\left(
			(a_{22}^++a_{22}^-)\frac{C_1}{e^{C_1\overline m}}
			+(a_{21}^+-a_{21}^-)\frac{\beta}{C_2+y_1}
			\right)v(y_1,0),
			\qquad y_1>0,
		\end{equation*}
		where $a_{ij}^+$ and $a_{ij}^-$ denote the traces of $a_{ij}$
		on the contact discontinuity $\Gamma^*$ from $\N^+$ and $\N^-$,
		respectively.
		Taking $C_2$ even larger if necessary gives
		\[
		\left|(a_{21}^+-a_{21}^-)\frac{\beta}{C_2+y_1}\right|
		\leqslant \frac{\lambda C_1}{e^{C_1\overline m}},
		\qquad y_1>0.
		\]
		Therefore,
		\[
		\bigl[a_{2j}\partial_{y_j}v\bigr]
		\leqslant
		-\frac{\lambda C_1}{e^{C_1\overline m}}\,v
		<0
		\qquad\text{on } \Gamma^*\cap\{y_1>0\}.
		\]
		For $y_1\geqslant2$, the definition of $g$ and
		\eqref{Cond:decay} give $|g(y_1,y_2)|\leqslant
		C(1+y_1)^{-\beta}$. On $0\leqslant y_1\leqslant2$,
		$g$ is bounded. Since the numerator of $v/C_3$ is at least
		one and $C_2$ is fixed, we may choose $C_3>0$, independently
		of $R$, so that
		\[
		v\geqslant |g| \quad\text{on } \partial\N_{0,R}.
		\]
		
		Since $v$ is continuous across $\Gamma^*$, the distribution
		$\partial_{y_i}(a_{ij}\partial_{y_j}v)$ contains a measure
		supported on the contact discontinuity $\Gamma^*$ with density
		$[a_{2j}\partial_{y_j}v]$. The conormal-jump term is negative, and the classical divergence
		is negative in each of $\N^+\cap\{y_1>0\}$ and
		$\N^-\cap\{y_1>0\}$. Thus, with the constants chosen above,
		\[
		\mathcal L_{\phi,\phi_r}(v\pm\psi^R)\leqslant0
		\quad\text{weakly in }\N_{0,R},\qquad
		v\pm\psi^R\geqslant0\quad\text{on }\partial\N_{0,R}.
		\]
		The weak maximum principle therefore yields
		\[
		\sup_{\N_{Y,R}}|\psi^R|
		\leqslant \sup_{\N_{Y,R}}v
		\leqslant \frac{C}{(1+Y)^\beta},
		\qquad 0<Y<R,
		\]
		where $C$ is independent of both $R$ and $Y$.
	\end{proof}
	
	We complete the proof of Theorem~\ref{Thm:asymptotic}.
	\begin{proof}[Proof of Theorem~\ref{Thm:asymptotic}]
		The limiting profile and the qualitative convergence were
		proved in Propositions~\ref{Pr:existence-r} and
		\ref{Pr:asymptotic-r}. It remains to prove the quantitative
		zero-order bound. By Proposition~\ref{Pr:decay-estimate},
		$\psi^R$ is uniformly bounded. The boundary lifting $g$ is
		fixed on each bounded substrip. A Caccioppoli estimate for
		$\psi^R-g$, with a cutoff away from the artificial end, gives
		a uniform $H^1$ bound on $\N_{0,L}$ for every fixed $L$ and
		all sufficiently large $R$. Hence a diagonal subsequence
		converges weakly in these $H^1$ spaces and strongly in local
		$L^2$ to a bounded weak solution $\psi^*$ of
		\[
		\begin{cases}
			\partial_{y_i}(a_{ij}\partial_{y_j}\psi^*)=0
			&\text{in }\N_{0,\infty},\\
			\psi^*=\zeta(W^\pm-a^\pm)
			&\text{on }\Gamma^\pm\cap\{y_1>0\},\\
			\psi^*=\psi&\text{on }\{y_1=0\}.
		\end{cases}
		\]
		The local $H^1$ convergence is weak, so continuity of the trace operator preserves the prescribed inlet and wall traces.
		
		To identify this limit, set $U:=\psi^*-\psi$ and take a cutoff
		$\eta_L(y_1)$ equal to one on $[0,L]$, zero for $y_1\geqslant2L$,
		and satisfying $|\eta_L'|\leqslant C/L$.
		The test function $\eta_L^2U$ has zero trace on the inlet and
		both walls. Ellipticity and Young's inequality give
		\[
		\int_{\N_{0,L}}|\nabla U|^2
		\leqslant C\int_{\N_{L,2L}}|\eta_L'|^2|U|^2
		\leqslant\frac{C}{L}\|U\|_{L^\infty(\N_{0,\infty})}^2.
		\]
		Letting $L\to\infty$ on each fixed substrip gives $\nabla U=0$; the zero boundary trace then gives $U=0$. Thus $\psi^*=\psi$. Passing the estimate of Proposition~\ref{Pr:decay-estimate} to the limit on bounded substrips and then exhausting the downstream tail yields
		\[
		\|\psi\|_{L^\infty(\N_{Y,\infty})}
		\leqslant C(1+Y)^{-\beta}.
		\]
		This proves \eqref{Est:beta} and completes the proof.
	\end{proof}
	
	\begin{proof}[Proof of Theorem~\ref{Thm:decay-C2}]
		Set $\psi:=\phi-\phi_r$. The coefficients in
		\eqref{Equ:psi} have the uniform ellipticity and
		$C^{1,\alpha}$ bounds established immediately after that
		equation. We use these bounds on translated strips.
		
		Choose $q_0>0$ so that $q-3\overline m\geqslant q/2>0$ for
		$q\geqslant q_0$. For such $q$, put
		\[
		\begin{aligned}
			K_q:={}&\|\psi\|_{L^\infty(\N_{q-3\overline m,q+3\overline m})}
			+|\zeta|\|W^\pm-a^\pm\|_{2,\alpha;
				(q-3\overline m,q+3\overline m)}.
		\end{aligned}
		\]
		By Theorem~\ref{Thm:asymptotic} and \eqref{Cond:decay-C2},
		\[
		K_q\leqslant C(1+q)^{-\beta}.
		\]
		We estimate $\psi$ on translated strips by combining the
		upper-wall and lower-wall boundary Schauder estimates with
		the Schauder estimate in \cite{Chen2024Elliptic}, using the
		overlapping neighborhoods from Proposition~\ref{Pr:app-estimate}.
		Their radii are fixed by $m^+$ and $m^-$, so the constants are
		independent of $q$.
		
		\medskip\noindent\textbf{Step 1. First derivative estimate.}
		Apply the two wall estimates and the contact-discontinuity estimate to \eqref{Equ:psi}. The outer neighborhoods lie in $\N_{q-3\overline m,q+3\overline m}$, and the translated three-strip cover gives
		\[
		\|\psi\|_{1,\alpha;\N^\pm_{q-2\overline m,q+2\overline m}}
		\leqslant C K_q.
		\]
		The boundary terms are bounded by $K_q$ because $\psi=\zeta(W^\pm-a^\pm)$ on $\Gamma^\pm$.
		
		\medskip\noindent\textbf{Step 2. Second derivative estimate.}
		Set $w:=\psi_{y_1}$. Tangential difference quotients give
		\[
		\partial_{y_i}(a_{ij}\partial_{y_j}w)=\partial_{y_i}F_i,
		\qquad F_i:=-(\partial_{y_1}a_{ij})\partial_{y_j}\psi,
		\]
		where $\mathbf F=(F_1,F_2)$. The wall condition is
		$w=\zeta(W^\pm)'$ on $\Gamma^\pm$, and
		\[
		[w]=0,\qquad [a_{2j}\partial_{y_j}w-F_2]=0
		\quad\text{on }\Gamma^*.
		\]
		The $C^{1,\alpha}$ coefficient bound and Step~1 imply
		\[
		\|\mathbf F\|_{0,\alpha;\N^\pm_{q-2\overline m,q+2\overline m}}
		\leqslant C K_q.
		\]
		Step~1 also bounds $\|w\|_{0,0;\N^\pm_{q-2\overline m,q+2\overline m}}$ by $C K_q$. We can therefore apply the upper-wall and lower-wall boundary estimates and the Schauder estimate in \cite{Chen2024Elliptic} to $w$.
		Combining them on the same three types of overlapping
		neighborhoods, now inside
		$\N_{q-2\overline m,q+2\overline m}$, gives
		\[
		\|w\|_{1,\alpha;\N^\pm_{q-\overline m,q+\overline m}}
		\leqslant C K_q.
		\]
		The sum of the wall terms is bounded by
		$|\zeta|\|(W^\pm)'\|_{1,\alpha;
			(q-2\overline m,q+2\overline m)}\leqslant K_q$.
		Thus $\psi_{y_1y_1}$ and $\psi_{y_1y_2}$ have the required
		$C^{0,\alpha}$ bounds. In each of $\N^+$ and $\N^-$,
		the equation gives
		\[
		a_{22}\psi_{y_2y_2}
		=-a_{11}\psi_{y_1y_1}-2a_{12}\psi_{y_1y_2}
		-(\partial_{y_i}a_{ij})\partial_{y_j}\psi.
		\]
		Since $a_{22}\geqslant\lambda>0$, this yields the same bound
		for $\psi_{y_2y_2}$. Consequently,
		\[
		\|\psi\|_{2,\alpha;\N^\pm_{q-\overline m,q+\overline m}}
		\leqslant C K_q\leqslant C(1+q)^{-\beta}.
		\]
		
		\medskip\noindent\textbf{Step 3. Estimate on the downstream tail.}
		Let $Y\geqslant q_0$. Applying the preceding estimate with
		$q=y_1$ gives the required supremum bounds for $\psi$ and its
		derivatives throughout $\N^+_{Y,\infty}$ and $\N^-_{Y,\infty}$.
		For two points in the same region at distance less than
		$\overline m/2$, choose $q$ to be the first coordinate of the
		first point. Both points belong to the corresponding strip,
		so the local $C^{2,\alpha}$ estimate bounds the H\"older
		quotient by $C(1+Y)^{-\beta}$. For pairs at distance at least
		$\overline m/2$, the supremum bound for the second derivatives
		gives the same estimate after division by
		$(\overline m/2)^\alpha$. This proves \eqref{Est:beta-C2} for
		$Y\geqslant q_0$. Increasing $C$ and using the global
		$C^{2,\alpha}$ bounds in both regions gives the result for
		$0<Y<q_0$.
	\end{proof}
	
	The decay estimate also gives the corresponding flow estimate in Lagrangian coordinates. Let $\mathbf U_r(y_2)$ be the state recovered from $\phi_r$ by \eqref{Def:flow-recovery}. Since $S,B$ are fixed and the interpolated gradients remain in a common compact subset of $\mathcal D$, the recovery formula and Theorem~\ref{Thm:decay-C2} give
	\begin{equation*}
		\|\mathbf U-\mathbf U_r\|_{1,\alpha;\N_{Y,\infty}^\pm}
		\leqslant C\|\phi-\phi_r\|_{2,\alpha;\N_{Y,\infty}^\pm}
		\leqslant C(1+Y)^{-\beta},\qquad Y>0.
	\end{equation*}
	Here the two states are compared at the same Lagrangian coordinates, separately in the upper and lower regions.
	
	\section{Acknowledgments}
	The research of Jun Chen and Xuemei Deng was partially supported by the Hubei Provincial Natural Science Foundation (Grant No.~2025AFB530). The work of Xiaoguang You was partially supported by the National Natural Science Foundation of China (Grant No.~12561038) and the Jiangxi Provincial Natural Science Foundation (Grant No.~20242BAB25008).

	\section*{Data sharing}
	Data sharing is not applicable to this article as no datasets were generated or analyzed during the current study.
	
	\section*{Conflict of interest}
	The authors declare that they have no financial or non-financial interests that are directly or indirectly related to the work submitted for publication.
	
	\section*{Declaration of AI Use}
	During the preparation of this work, the authors used DeepSeek in order to improve language and readability. DeepSeek was used for language polishing, including grammar correction, terminology standardization, and translation proofreading.

\end{document}